\documentclass[11pt]{amsart}
\usepackage{amssymb, latexsym,graphicx,rotating}

\numberwithin{equation}{section}
\newtheorem{Thm}[equation]{Theorem}
\newtheorem{Prop}[equation]{Proposition}
\newtheorem{Cor}[equation]{Corollary}
\newtheorem{Lem}[equation]{Lemma}

\theoremstyle{definition}
\newtheorem{Def}[equation]{Definition}

\newtheorem{Rmk}[equation]{Remark}

\begin{document}

\title [VOA associated to type $G$ affine Lie algebras]
{Vertex operator algebras \\ associated to type $G$ affine Lie algebras II}
\author[J. D. Axtell]{Jonathan D. Axtell}
\thanks{This work was supported by NRF Grant \# 2010-0010753.}
\address{Department of Mathematics, Seoul National University, 599 Gwanak-ro, Gwanak-gu, Seoul 151-747, Korea}
\email{jdaxtell@snu.ac.kr}

\begin{abstract}
We continue the study of the vertex operator algebra $L(k,0)$ associated to a type $G_2^{(1)}$ affine Lie algebra at admissible one-third integer levels, $k = -2 + m + \tfrac{i}{3}\ (m\in \mathbb{Z}_{\ge 0}, i = 1,2)$, initiated in \cite{AL}.  Our main result is that there is a finite number of irreducible $L(k,0)$-modules from the category $\mathcal{O}$.
The proof relies on the knowledge of an explicit formula for the singular vectors.  After obtaining this formula, we are able to show that there are only finitely many irreducible $A(L(k,0))$-modules form the category $\mathcal{O}$.  The main result then follows from the bijective correspondence in $A(V)$-theory.
\end{abstract}

\maketitle

\section*{Introduction}

In this paper, we continue the study of vertex operator algebras associated to a Lie algebra of type $G_2^{(1)}$ at admissible one-third integer levels, which was undertaken in \cite{AL}.  We briefly recall the formulation.

Vertex operator algebras (VOA) are mathematical counterparts of chiral algebras in conformal field theory.  An important family of examples comes from representations of affine Lie algebras.  More precisely, if we let $\hat{\mathfrak{g}}$ be an affine Lie algebra, the irreducible $\hat{\mathfrak{g}}$-module $L(k,0)$ of highest weight $k \Lambda_0$, $k\in \mathbb{C}$, is a VOA whenever $k \neq - h^{\vee}$, the negative of the dual Coxeter number.

The representation theory of $L(k,0)$ varies depending on the values $k \in \mathbb{C}$.  If $k$ is a positive integer, the VOA $L(k,0)$ has only finitely many irreducible modules which coincide with the integrable highest weight $\hat{\mathfrak{g}}$-modules of level $k$, and the category of $\mathbb{Z}_+$-graded weak $L(k,0)$-modules is semisimple.  For generic $k \in \mathbb{C}$, however, the behavior is much different.  (For example, see \cite{KL1, KL2}.)  Yet for certain values $k\in \mathbb{Q}$ satisfying $k+h^{\vee} >0$, the category of weak $L(k,0)$-modules which belong to the category $\mathcal{O}$ as $\hat{\mathfrak{g}}$-modules has structure similar to the category of $\mathbb{Z}_+$-graded weak modules for positive integer values.  Such values $k \in \mathbb{Q}$ are called {\em admissible levels}.  This notion was defined in the works of Kac and Wakimoto \cite{KW1, KW2}.

Several cases have been studied in varying degrees of generality.  In \cite{A1}, Adamovi{\' c} studied the case of admissible half-integer levels for type $C_l^{(1)}$.  The case of all admissible levels of type $A_1^{(1)}$ has been studied by Adamovi{\' c} and Milas in \cite{AM}, and by Dong, Li and Mason in \cite{DLM}.  More recently  in  \cite{P,P1}, Per{\v s}e studied admissible half-integer levels for type $A_1^{(1)}$ and $B_l^{(1)}$.  In \cite{AL}, the author and Lee studied admissible one-third integer levels for type $G_2^{(1)}$.  

In this progression, a fundamental r{\^ o}le is played by the $A(V)$-theory.  The associative algebra $A(V)$ attached to a vertex operator algebra $V$ was introduced by Frenkel and Zhu in \cite{FZ,Z}, where it was shown that irreducible $\mathbb{Z}_+$-graded weak $V$-modules are in bijective correspondence with irreducible $A(V)$-modules.  In the affine case, it is useful to consider the restriction of this correspondence to the subcategories listed in the following diagram:
 
\begin{center}
\begin{tabular}{| c | c |}
\hline
 $L(k,0)$  & $A(L(k,0))$ \\
\hline
& \\
\{$\mathbb{Z}_+$-graded weak modules\}  &  \{Modules\}  \\ 
\begin{sideways}$\subseteq$\end{sideways} & \begin{sideways}$\subseteq$\end{sideways} \\
\{Weak modules from the category $\mathcal{O}$\}  &  \{Modules from the category $\mathcal{O}$\}\\
\begin{sideways}$\subseteq$\end{sideways} & \begin{sideways}$\subseteq$\end{sideways}\\
\{Modules\}  & \{Finite dimensional modules\}\\
&\\
\hline
\end{tabular} 
\end{center}

\noindent The set of isomorphism classes of irreducible  objects from each category on the left-hand side is in bijective correspondence with the classes of irreducible objects from the respective category on the right-hand side.  If $k$ is a positive integer, then $L(k,0)$ is a {\em rational} VOA.  Hence there is no difference between modules and $\mathbb{Z}_+$-graded weak modules (see \cite{Z}), and the three categories of modules collapse to a single category.

For non-integer admissible levels, however, there is a distinction.  
While there are infinitely many classes of irreducible $\mathbb{Z}_+$-graded weak modules, it was shown for all cases considered in the works mentioned above that there is only a finite number of classes of irreducible weak $L(k,0)$-modules from the category $\mathcal{O}$.  Furthermore, it was conjectured in \cite{AM} that such modules are completely reducible for all admissible levels $k$.  I.e., it was conjectured that the vertex operator algebra $L(k,0)$ is {\em rational in the category $\mathcal{O}$}.  This was also verified in most of the cases considered above.  In particular, it was shown in \cite{AL} that $L(k,0)$ is rational in the category $\mathcal{O}$ for admissible levels, $k= -\frac{5}{3}, -\frac{4}{3}, -\frac{2}{3}$, for type $G_2^{(1)}$.

In this paper, we show there are only finitely many irreducible weak $L(k,0)$-modules from the category $\mathcal{O}$ for all admissible one-third integer levels, $k= -2 + m +\frac{i}{3}$ ($m \in \mathbb{Z}_{\ge 0}, i=1,2$), for type $G_2^{(1)}$.
By a result in \cite{KW1}, $L(k,0)$ is the irreducible quotient of a generalized Verma module with maximal ideal generated by a vector $v_k$.  The method of classification using Zhu's theory requires an explicit formula for this singular vector.  While this formula was verified by direct calculation in \cite{AL}, the number of computations involved increases rapidly as $k$ increases.  For levels $k\ge -\frac{1}{3}$, an alternative approach seems necessary.

Results in \cite{KK, KW1} guarantee the existence of a singular vector $v_k$ belonging to a certain weight space in the generalized Verma module.  As it turns out,
the defining properties of a singular vector together with this weight condition are quite strong.  We will find a formula for $v_k$ by describing the set of vectors which satisfy all such conditions. 
Exploiting an isomorphism between the generalized Verma module and a universal enveloping algebra, we begin in Section 2 by transferring some of the conditions for $v_k$ to this associative algebra.  
In particular, it will follow from Lemma \ref{lem:P3:universal} that $v_k$ is the image of a monomial of certain elements $u,v,w$ in the algebra.  This greatly simplifies our computations, and in Section 3, we complete the description of the singular vectors (see Proposition \ref{prop:sing} and Theorem \ref{thm:sing}).

We then use the singular vectors to obtain information about the representation theory of $A(L(k,0))$.  By considering certain polynomials in the symmetric algebra of the Cartan subalgebra, we prove in Section 4 that there are only finitely many irreducible $A(L(k,0))$-modules from the category $\mathcal{O}$ for each admissible one-third integer level, $k= -2 + m +\frac{i}{3}$.  Our main result, Theorem \ref{thm:main}, which states that there are only finitely many irreducible weak $L(k,0)$-modules from the category $\mathcal{O}$, then follows from the bijective correspondence in $A(V)$-theory.

\subsection*{Acknowledgments} 
The author wishes to thank K.-H. Lee and A. Feingold for helpful conversations and encouragement.

\bigskip

\section{Preliminaries}

In this section, we collect several facts needed in the remainder.  We retain the notation given in the previous paper \cite{AL}.  The reader may refer to the preliminary section of that paper for definitions and further references.

Given an affine Lie algebra $\hat{\frak g }$, recall that $\hat{\frak g}_{\pm} = \frak g \otimes t^{\pm1} \mathbb C[t^{\pm1}]$.  Let $\mathcal{U}(\cdot )$ and $\mathcal{S}( \cdot )$ denote the universal enveloping algebra and symmetric algebra, respectively.  The adjoint action of the finite-dimensional Lie subalgebra $\frak g$ on $\hat{\frak g}$ is locally finite and leaves $\hat{\frak g}_{\pm}$ invariant.  We thus have a corresponding locally finite action of $\frak g$ on $\mathcal{U}(\hat{\frak g})$ (resp. $\mathcal{S}(\hat{\frak g}$)) which restricts to an action on $\mathcal{U}(\hat{\frak g}_{\pm})$ (resp. $\mathcal{S}(\hat{\frak g}_{\pm}$)).   From results in \cite{D}, it may be seen that there is an isomorphism of $\frak g$-modules, $\omega: \mathcal{S}(\hat{\frak g}_-) \cong \mathcal{U}(\hat{\frak g}_-)$, defined by setting
\begin{equation}\label{eq:isom1} \omega(X_1 \cdots X_n) = \frac{1}{n!}\sum_{\sigma \in S_n} X_{\sigma(1)}\cdots X_{\sigma(n)}
\end{equation}
for $X_1, \dots, X_n \in \hat{\frak g}_-$, where $S_n$ is the symmetric group.  Note here that we will generally use upper case letters (e.g. $X$, $Y$, etc.) for elements of a Lie algebra and lower case letters (e.g. $x$, $y$, etc.) for elements of an associative algebra.

Given $\mu \in \frak h^*$ and $k \in \mathbb{C}$, recall that the {\em generalized Verma module} \[ N(k,\mu) = \mathcal{U}(\hat{\frak g})\otimes_{\mathcal{U}(\frak g \oplus \mathbb C K \oplus \hat{\frak g}_+)} V(\mu)\]
 is the $\hat{\frak g}$-module induced from the irreducible $\frak g$-module $V(\mu)$, where $\hat{\frak g}_+$ acts trivially and $K$ acts as scalar multiplication by $k$.  The {\em vacuum} module $N(k,0)$ is a VOA whenever $k \neq -h^{\vee}$.

Since the adjoint action of $\frak g$ leaves the subspace $\frak g \oplus \mathbb C K \oplus \hat{\frak g}_+ \subseteq \hat{\frak g}$ invariant, we have an induced action of $\frak g$ on $N(k,0)$.  Furthermore, there is an isomorphism $\varphi: \mathcal{U}(\hat{\frak g}_-) \cong N(k,0)$ of $\frak g$-modules, given by
\begin{equation}\label{eq:isom2} \varphi(x) = x.\mathbf{1} \quad (x \in \mathcal{U}(\hat{\frak g}_-) ),
\end{equation} 
where $\mathbf{1} \in N(k,0)$ is the vacuum vector.

Algebras $\mathcal{U}(\hat{\frak g}_-)$ and $\mathcal{S}(\hat{\frak g}_-)$ are both naturally graded by the affine root lattice $\widehat{Q}$: \[ \mathcal{U}(\hat{\frak g}_-) = \bigoplus_{\hat{\gamma} \in \widehat{Q} } \mathcal{U}(\hat{\frak g}_-)_{ (\hat{\gamma}) }, \quad \mathcal{S}(\hat{\frak g}_-) = \bigoplus_{\hat{\gamma} \in \widehat{Q} } \mathcal{S}(\hat{\frak g}_-)_{ (\hat{\gamma}) }. \]
On the other hand, $N(k,0)$ is graded by the coset $k\Lambda_0 + \widehat{Q}  \subseteq \hat{\frak h}^*$:
\[ N(k,0) = \bigoplus_{\hat{\gamma} \in \widehat{Q} } N(k,0)_{ (k\Lambda_0 + \hat{\gamma}) }. \]
Additionally,  these gradings are all compatible with the respective $\frak g$-module action.  Let $wt(\cdot)$ denote the weight of an homogeneous element with respect to any of the above gradings.  The isomorphism $\omega: \mathcal{S}(\hat{\frak g}_-) \rightarrow \mathcal{U}(\hat{\frak g}_-)$ preserves the grading: if $x \in \mathcal{S}(\hat{\frak g}_-)$ is homogeneous, then
\begin{equation}\label{eq:grading1}
wt (\omega (x) ) = wt(x).
\end{equation}
While for the isomorphism $\varphi: \mathcal{U}(\hat{\frak g}_-) \rightarrow N(k,0)$, we clearly have: if $y \in \mathcal{U}(\hat{\frak g}_-)$ is homogeneous, then
\begin{equation}\label{eq:grading2}
wt(\varphi(y)) = wt(y) + k \Lambda_0.
\end{equation}

From now on, we assume that $\hat{\mathfrak{g}}$ is a Lie algebra of type $G_2^{(1)}$.  Recall that the root vectors of the finite dimensional subalgebra $\mathfrak g$ are denoted $\{ E_{ij} , F_{ij}\ |\ i\alpha + j \beta \in \Delta_+\}$ and that we write $H_{ij} = [E_{ij}, F_{ij}] = (i \alpha + j \beta)^{\vee}$. 

We will write $\theta^{\vee} = 2\alpha + \beta$ to denote the highest short root in $\Delta$.  Recall from Lemma 2.2 of \cite{AL} that the {\em vacuum} admissible weights at one-third integer levels for type $G_2^{(1)}$ have the form \[ \lambda_{3m+i} = (m-2+\tfrac{i}{3})\Lambda_0\quad (m \in \mathbb{Z}_{\ge 0}, i = 1,2). \]

The following lemma describes essential properties of the singular vector whose formula we seek.
\begin{Lem}\label{lem:exist}
Suppose $m \in \mathbb{Z}_{\ge 0}$ and $i \in \{1,2\}$.  Let $\lambda_{3m+i} = (m-2+\frac{i}{3})\Lambda_0$, $k=m-2+\tfrac{i}{3}$ and $n=3k+7$. Then there exists a singular vector $v_k \in N(k,0)$ such that the maximal $\hat{\frak{g}}$-submodule $J(k,0)$ of $N(k,0)$ is generated by $v_k$.  Furthermore, $v_k$ satisfies the following properties:
\begin{align*}
\phantom{LLLLL}&(P1) & wt(v_k) =&\ k \Lambda_0 +  n\theta^{\vee} - n \delta, &\phantom{LLLLL} &(P2) & E_{01}(0).v_k =&\ 0,&\phantom{LLLLLLLLLL}\\
&(P3) & E_{10}(0).v_k =&\ 0,& &(P4) & F_{32}(1).v_k =&\ 0.&
\end{align*}
\end{Lem}

\begin{proof}
From Lemma 2.2 of \cite{AL}, we have: \[ \widehat{\Pi}_{ \lambda_{3n+i} }^{\vee} = \{ (\delta - \theta^{\vee})^{\vee}, \alpha^{\vee}, \beta^{\vee} \}. \]
It is implicit in Proposition 1.3 of \cite{AL} that the maximal submodule of the Verma module $M(\lambda_{3m+i})$ is generated by three nonzero singular vectors with weights
\[r_{\delta-\theta^{\vee}}\cdot \lambda_{3m+i},\quad r_{\alpha}\cdot \lambda_{3m+i},\quad r_{\beta}\cdot \lambda_{3m+i},\]
respectively. For the first of these weights, we have by Lemma 2.2:
\begin{align*}
r_{\delta-\theta^{\vee}} \cdot \lambda_{3m+i}\ & =  \ \lambda_{3m+i} - \langle \lambda_{3m+i} + \rho, (\delta - \theta^{\vee})^{\vee}\rangle (\delta - \theta^{\vee}) \\  
 & =  \  \lambda_{3m+i} - (3m+i+1)(\delta - \theta^{\vee}) \\
 & =  \  \lambda_{3m+i} - n(\delta - \theta^{\vee}). 
\end{align*}
While for the other weights, we have:
\[ r_{\alpha} \cdot \lambda_{3m+i}\ =\ \lambda_{3m+i} - \langle \lambda_{3m+i}+\rho, \alpha^{\vee} \rangle \alpha\ =\ \lambda_{3m+i} - \alpha,\quad \mbox{and}\]
\[ r_{\beta} \cdot \lambda_{3m+i}\ =\ \lambda_{3m+i} - \langle \lambda_{3m+i}+\rho, \beta^{\vee} \rangle \beta\ =\ \lambda_{3m+i} - \beta,\] respectively.  It is thus clear from the definition of $N(k,0)$ that the singular vectors in $M(\lambda_{3m+i})$ corresponding to the last two weights both map to zero under the natural projection onto $N(k,0)$, while the singular vector in $M(\lambda_{3m+i})$ corresponding to the first weight must map to a nonzero singular vector, $v_{k}$, of weight $k\Lambda_0 + n\theta^{\vee} - n\delta$, which is property $(P1)$.  Since $\hat{\mathfrak{n}}_+$ is generated by the elements $E_{10}(0), E_{01}(0)$, and $F_{32}(1)$, the vector $v_{k}$ also satisfies properties $(P2)-(P4)$ by the definition of singular vector.
\end{proof}

In Sections 2 and 3, we will describe the subspace of elements in $N(k,0)$ which satisfy properties $(P1) - (P4)$ in the above lemma.  For the first three properties $(P1)-(P3)$ it will be convenient to work with the associative algebras $\mathcal{U}(\hat{\frak g }_-)$ and $\mathcal{S}(\hat{\frak g }_-)$.

The property $(P1)$ says that $v_k \in N(k,0)_{ (k\Lambda_0 + n \theta^{\vee} - n \delta) }$, where $n = 3k+7$.  It follows from (\ref{eq:grading2}) that we must have $v_k = x.1$, for some $x \in \mathcal{U} ( n \theta^{\vee} - n \delta )$.  We may then consider the subalgebra $\mathcal{U}_{\theta^{\vee}-\delta} \subseteq \mathcal{U}(\hat{\frak g}_-)$ defined by  \[ \mathcal{U}_{\theta^{\vee} - \delta} = \bigoplus_{n \in \mathbb{Z}_{\ge 0} } \mathcal{U}(\hat{\frak g}_-)_{ (n \theta^{\vee} - n \delta) }.\]
It will also be convenient to consider the corresponding commutative analogue, $\mathcal{S}_{\theta^{\vee}-\delta} \subseteq \mathcal{S}(\hat{\frak g}_-)$, defined by \[\mathcal{S}_{\theta^{\vee}-\delta} = \bigoplus_{n \in \mathbb{Z}_{\ge 0} } \mathcal{S}(\hat{\frak g}_-)_{ (n \theta^{\vee} - n \delta) }.\] It is not difficult to see that $\omega( \mathcal{S}_{\theta^{\vee}-\delta} ) = \mathcal{U}_{\theta^{\vee}-\delta} $.  In Section 2, we will show that $\mathcal{S}_{\theta^{\vee}-\delta}$ is a polynomial algebra in infinitely many variables, and we also describe the subalgebras of $\mathcal{S}_{\theta^{\vee}-\delta}$ and $\mathcal{U}_{\theta^{\vee}-\delta}$ satisfying properties corresponding to $(P2)$ and $(P3)$.  In Section 3, we use this information, along with the map $\varphi: \mathcal{U}(\hat{\mathfrak g}_-) \rightarrow N(k,0)$, to give a description of the subspace of $N(k,0)$ satisfying all four properties $(P1)-(P4)$, yielding the desired formula for the singular vector.

\bigskip

\section{Subalgebras related to singular vectors}

In this section we first describe each subalgebra of $\mathcal{S}(\hat{\frak g }_-)$ in the sequence
\[
(\mathcal{S}_{\theta^{\vee}-\delta})^{\frak n_+}\ \subseteq \quad (\mathcal{S}_{\theta^{\vee}-\delta})^{\beta}\ \subseteq \quad \mathcal{S}_{\theta^{\vee}-\delta}\ \subseteq\quad \mathcal{S}(\hat{\frak g }_-),
\]
corresponding to the properties $(P1)$, $(P2)$, and $(P3)$  of Lemma \ref{lem:exist}, respectively.  We will then use these descriptions together with the fact that $(\mathcal{S}_{\theta^{\vee}-\delta})^{\frak n_+} \cong (\mathcal{U}_{\theta^{\vee}-\delta})^{\frak n_+}$ to describe the subalgebra $(\mathcal{U}_{\theta^{\vee}-\delta})^{\frak n_+} \subseteq \mathcal{U}(\hat{\frak g }_-)$.

We fix the following ordered basis of $\mathfrak{g}$:
\begin{equation}\label{eq:basis} 
(E_{32}, E_{31}, E_{21}, E_{11}, E_{10}, E_{01},  H_{01}, F_{01}, H_{21}, F_{10}, F_{11}, F_{21}, F_{31}, F_{32}).
\end{equation}
We have used $H_{21} = (2\alpha + \beta)^{\vee}$  $( = 2H_{10} + 3H_{01})$ in the above list so that our basis can be partitioned into subbases for invariant subspaces of the $\mathfrak{sl}_2$-triple \[ \mathfrak{g}_{\beta}\ =\ \mathbb{C} E_{01} \oplus \mathbb{C} H_{01} \oplus \mathbb{C} F_{01}. \]
Let us fix a basis for $\hat{\mathfrak{g}}_-$ which consists of the set of all $X(-n)$ such that $X$ is a basis element from (\ref{eq:basis}) and $n> 0$.  This basis is ordered so that $X(-n) \le X'(-n)$ if $X \le X'$, and  $X(-m) < X'(-n)$ whenever $m<n$.  Let $\mathcal{B}_{PBW}$ denote the corresponding PBW basis of $\mathcal{U}(\hat{\mathfrak g}_-)$, which may also be considered a basis for $\mathcal{S}(\hat{\frak g}_-)$.

The following subsets form a partition of the $\mathfrak{g}$-basis (\ref{eq:basis}):
\begin{align*}
& B^1 = \{E_{32}, E_{31} \}, & \quad & B^2 = \{E_{21} \}, & \quad & B^3 = \{E_{11},E_{10} \}, & \quad & B^4 = \{E_{01}, H_{01}, F_{01} \}, \\
& B^5 = \{H_{21} \}, & \quad & B^6 = \{F_{10}, F_{11} \}, & \quad & B^7 = \{F_{21} \}, & \quad & B^8 = \{F_{31}, F_{32} \}.
\end{align*}
We may similarly decompose the basis for $\hat{\mathfrak{g}}_-$ as the union of the following subbases: 
\begin{align*}
& B_{-j}^1 = \{E_{32}(-j), E_{31}(-j) \}, & & B_{-j}^2 = \{E_{21}(-j) \}, &  & B_{-j}^3 = \{E_{11}(-j),E_{10}(-j) \},\\
 &  B_{-j}^4 = \{E_{01}(-j), H_{01}(-j), F_{01}(-j) \}, &  & B_{-j}^5 = \{H_{21}(-j) \}, & & B_{-j}^6 = \{F_{10}(-j), F_{11}(-j) \},\\
  & B_{-j}^7 = \{F_{21}(-j) \}, & & B_{-j}^8 = \{F_{31}(-j), F_{32}(-j) \}& & \ ( j \in \mathbb{Z}_{\ge 1}). 
\end{align*}
Together these collections give decompositions,
\[ 
\frak g\ = \ \bigoplus_{ i \in \{1, \dots, 8\} } \mbox{Span}_{\mathbb C}( B^{i}) \ \mbox{ and } \quad
\hat{\frak g}_-\ = \ \bigoplus_{ \substack{i \in \{1, \dots, 8\} \\ j \in \mathbb{Z}_{\ge 1} } } \mbox{Span}_{\mathbb C} (B^{i}_{-j}),
\]
of $\frak g$ and $\hat{\frak g}_-$, respectively, into irreducible $\mathfrak{g}_{\beta}$-submodules.

In the case of the symmetric algebra, the PBW basis  may be partitioned into subbases of {\em invariant} subspaces for the action of $\mathfrak{g}_{\beta}$ on $\mathcal{S}(\hat{\frak g}_-)$.
\begin{Def}
A {\em $\beta$-string} is a set $B \subseteq \mathcal{S}(\hat{\frak g}_-)$ of the form
\[ B = B^{i_1}_{-j_1} B^{i_2}_{-j_2} \cdots B^{i_t}_{-j_t} = \{X_1 X_2 \cdots X_t\ |\ X_l \in B^{i_l}_{-j_l} \mbox{ for } l=1,\dots, t \}, \] where $i_1, \dots, i_t \in \{1,\dots,8 \}$ and $j_1, \dots, j_t \in \mathbb{Z}_{\ge 1}$.  A $\beta$-string with only one factor is called a {\em simple $\beta$-string}, and the set $\{1\}$ containing the identity is a trivial $\beta$-string.
\end{Def}

The collection of all $\beta$-strings, $B$, forms a partition of $\mathcal{B}_{PBW}$, and we have a decomposition of $\mathcal{S}(\hat{\mathfrak g}_-)$ into $\mathfrak{g}_{\beta}$-invariant subspaces:  $\mathcal{S}(\hat{\mathfrak g}_-) = \oplus_{B}\mbox{ Span}_{\mathbb{C}}(B)$.  Now recall the subalgebra $\mathcal{S}_{\theta^{\vee}-\delta}$ introduced in Section 1.  It is not difficult to see that 
\[
\mathcal{S}_{\theta^{\vee}-\delta} \cap \mathcal{B}_{PBW} = \{ y \in \mathcal{B}_{PBW}\ |\ wt(y) = n\theta^{\vee}-n\delta, \mbox{ for some }n \in \mathbb{Z}_{\ge 0}  \}
\]
 forms a basis of $\mathcal{S}_{\theta^{\vee}-\delta}$.  This basis may similarly be partitioned into subsets $\mathcal{S}_{\theta^{\vee}-\delta} \cap B$.  In the next lemma, we will explicitly describe each set $\mathcal{S}_{\theta^{\vee}-\delta} \cap B$ corresponding to an arbitrary $\beta$-string $B$.  In order to give this description, it will be convenient to factor out the largest power of $B^1_{-1}$ from a $\beta$-string: we say $B=B^{i_1}_{-j_1} \cdots B^{i_t}_{-j_t}$ {\em has no $B^1_{-1}$-factors} if  $B^{i_l}_{-j_l} \neq B^1_{-1}$ for $l=1,\dots, t$.

Given a $\beta$-string, $B=B^{i_1}_{-j_1} \cdots B^{i_t}_{-j_t}$, with no $B^1_{-1}$-factors and an element $y=y_1\cdots y_t \in B$, with $y_l \in B^{i_l}_{-j_l}\ ( l=1,\dots,t)$, we write
\begin{equation}\label{eq:def}
m_i(y) = m_i(y_1) + \cdots + m_i(y_t)\ (i=1,2) \quad \mbox{ and } \quad n(y) = n(y_1)+ \cdots + n(y_t),
\end{equation}
where the integers $m_i(y_1), \dots, m_i(y_t)$ and $n(y_1), \dots, n(y_t)$ are given in Table \ref{table:one}.  We also write $m(B) = m_1(y)+m_2(y)$, which is well-defined since, from (\ref{eq:def}) and Table \ref{table:one}, $m_1(y)+m_2(y)=m_1(y')+m_2(y')$ for $y,y' \in B$.

 \begin{table}
 \begin{tabular}{| c || c | c | c | c |}
\hline
\bf $ \beta$-string & $\bf y$ &  $\bf m_1$ & $\bf m_2$ & $\bf n$\\
\hline
 \hline
 $B^1_{-j-1} $& $E_{32}(- j-1)$ & $j$& $j+1$ & $3j $  \\
  				&$E_{31}(-j-1)$ & $j+1$ & $j$  & "  \\
\hline
 $B^2_{-j} $&$E_{21}(-j)$ & $j$ & $j$ & $3j-2 $  \\
& & & &\\
\hline
 $B^3_{-j}$& $E_{11}(-j)$ &  $j$  & $j+1$ & $3j-1 $   \\
 				&  $E_{10}(-j)$ & $j+1$ & $j$   & "   \\
\hline
 $B^4_{-j}$&  $E_{01}(-j)$& $j$  & $j+2$ & $3j $   \\
				&  $H_{01}(-j)$ & $j+1$ & $j+1$   & "    \\
				&  $F_{01}(-j)$ & $j+2$ & $j$   &  " \\
\hline
\end{tabular}
\ 
 \begin{tabular}{| c || c | c | c | c |}
\hline
\bf $ \beta$-string & $\bf y$ &  $\bf m_1$ & $\bf m_2$ & $\bf n$\\
\hline
\hline
 $B^5_{-j}  $ & $H_{21}(-j)$ & $j+1$ & $j+1$ & $3j $  \\
& & & &\\
\hline
 $B^6_{-j}	$	& $F_{10}(-j)$ & $j+1$ & $j+2$  & $3j+1$  \\
			 & $F_{11}(-j)$ & $j+2$ & $j+1$ &   " \\
\hline
 $B^7_{-j} $	&  $F_{21}(-j)$ &  $j+2$  & $j+2$ & $3j+2 $   \\
& & & &\\
\hline			
$B^8_{-j}	$	&  $F_{31}(-j)$ & $j+2$ & $j+3$   & $3j+3$  \\
			&  $F_{32}(-j)$ & $j+3$ & $j+2$   &  " \\
& & & &\\			
\hline

\end{tabular}
 \medskip
 \caption{Values of $m_1,m_2, n$ for simple $\beta$-strings $B^{i}_{-j} \neq B^{1}_{-1}\ (j \in \mathbb{Z}_{\ge 1})$.}
 \label{table:one}
\end{table}

\begin{Lem}\label{lem:P1}
Suppose $B=B^{i_1}_{-j_1} \cdots B^{i_t}_{-j_t}$ is a $\beta$-string with no $B^1_{-1}$-factors, and let $m\in \mathbb{Z}_{\ge 0}$.  Then
\begin{equation*}
\mathcal{S}_{\theta^{\vee}-\delta} \cap \left((B^1_{-1})^m B\right) =
\begin{cases} \{ E_{32}(-1)^{m_1(y)} E_{31}(-1)^{m_2(y)} y\ |\ y \in B \} & \text{if $m=m(B)$}
\\
\quad \emptyset &\text{otherwise.}
\end{cases}
\end{equation*}
Furthermore, $\mathcal{S}_{\theta^{\vee}-\delta}$ is a polynomial algebra with an infinite number of generators.
\end{Lem}

\begin{proof}
Let $y \in B$ and suppose $wt(y) = c_1 \alpha + c_2 \beta + c_3 \delta$ for some $c_i \in \mathbb{Z}\ (i=1,2,3)$.  We wish to determine for which integers, $m_1, m_2, n \in \mathbb{Z}$, we have
\[wt(E_{32}(-1)^{m_1}E_{31}(-1)^{m_2} \cdot y)\ =\ n \theta^{\vee} - n \delta, \]
which is true if and only if
\begin{equation}\label{eq:soln1} m_1 (3\alpha+2\beta-\delta)\ +\ m_2(3\alpha+\beta-\delta)\ +\ c_1\alpha + c_2\beta +c_3\delta\ =\ n(2\alpha+\beta-\delta).
\end{equation}
By comparing the coefficients of $\alpha, \beta$, and $\delta$, respectively, on both sides of (\ref{eq:soln1}), we may obtain the system of equations:
\begin{align}
 3m_1 + 3m_2 - 2n\ =&\ \ -c_1\label{eq:system}\\
 2m_1 +  \phantom{3}m_2 - \phantom{2}n\ =&\ \ -c_2\nonumber \\
 -m_1 -  \phantom{3}m_2 + \phantom{2}n\ =&\ \ -c_3\nonumber
\end{align}
This may be regarded as a system of three equations in three unknowns corresponding to $m_1, m_2, n$.  Thus the system has at most one solution.  

Using Table \ref{table:one} and the formulas (\ref{eq:def}), we see that setting $m_i=m_i(y)$, for $i=1,2$, and $n=n(y)$ yields a solution to the equation such that $m_1+m_2 = m(B)$. We have thus arrive at the unique solution to the system, proving the first statement of the lemma.
To verify the second statement of the lemma, notice that the set
\[ \{ E_{32}(-1)^{m_1(y)}E_{31}(-1)^{m_2(y)} y\ |\ y \mbox{ belongs to a simple } \beta\mbox{-string } B \neq B^1_{-1}\}\quad \subseteq \quad \mathcal{B}_{PBW}
\]
forms an algebraically independent set of generators for $\mathcal{S}_{\theta^{\vee} - \delta}$. 
\end{proof}

Let $B=  B^{i_1}_{-j_1} \cdots B^{i_t}_{-j_t}$ be an arbitrary $\beta$-string.  Recall that Span$_{\mathbb{C}} B$ is a $\frak g_{\beta}$-invariant space.  We may then define constants $c(y,y') \in \mathbb{C}$, for $y,y' \in B$, by the equation
\[
E_{01}.y = \sum_{y' \in B} c(y,y')\ y'.
\]
These coefficients define a partial order on $B$.  First define a covering relation, $\vdash$, by $y' \vdash y$  if $c(y,y') \neq 0$.  Then define a partial order, $\preceq$, by setting
$y \preceq y'$ if there exist elements $y_1, \dots, y_t \in B$ such that $y' = y_t \vdash y_{t-1} \vdash  \dots \vdash y_1 = y$.

\begin{Rmk} \label{rmk:minimum}  
Let $B\neq B^{1}_{-1}$ be a simple $\beta$-string.  Then $B$ has a unique minimal element, $y_- = y_-(B)$, with respect to the partial order $\preceq$.  It is clear from Table \ref{table:one} that $m_2(y')=m_2(y)+1$ whenever $y' \vdash y$ ($y,y' \in B)$, and
$y_-$ may be characterized as the unique element of $B$ for which the value of $m_2(\cdot)$ is minimized.  It is also easy to check from Table \ref{table:one} that $m_2(y_-(B)) = 0$ if  $B \in \{ B^2_{-1}, B^3_{-1}, B^4_{-1}, B^1_{-2}\}$, and $m_2(y_-(B))>0$ for all other simple $\beta$-strings $B \neq B^1_{-1}$.

In general, suppose $B=B^{i_1}_{-j_1} \dots B^{i_t}_{-j_t}$ is a $\beta$-string with no $B^1_{-1}$-factors.  If we write $y^l_-=y_-(B^{i_l}_{-j_l})\ (l=1, \dots, t)$, then it follows from the definition of the adjoint action of $\mathfrak{g}$ on $\mathcal{S}(\hat{\mathfrak g}_-)$ that
\[
y_-\ =\ y_-(B)\ =\ y^1_- y^2_- \cdots y^t_-
\]
is the unique minimal element of $B$.  It follows from (\ref{eq:def}) and the previous paragraph that we also have 
\begin{equation}\label{eq:cover}
m_1(y') = m_1(y)-1 \quad \mbox{and}\quad m_2(y')=m_2(y)+1,\mbox{ whenever } y' \vdash y\ (y,y' \in B),
\end{equation}
 and $y_-$ may be characterized as the unique element of $B$ for which $m_2(\cdot)$ is minimized.  Additionally, it may also be checked that $m_2(y_-(B)) = 0$  if and only if $B= (B^2_{-1})^p (B^3_{-1})^q (B^4_{-1})^r (B^1_{-2})^s$ for some $(p,q,r,s) \in (\mathbb{Z}_{\ge 0})^4$.
\end{Rmk}

In the remainder, we fix the following elements of $\mathcal{S}(\hat{\mathfrak g}_-)$.  Set 
\begin{align*}
& a = E_{21}(-1), \qquad b =  E_{31}(-1) E_{11}(-1)\   -\ E_{32}(-1) E_{10}(-1), \\
& c = E_{31}(-1)^2 E_{01}(-1)  - E_{32}(-1) E_{31}(-1) H_{01}(-1)    - E_{32}(-1)^2 F_{01}(-1),  \mbox{ and}\\
& w = E_{31}(-1) E_{32}(-2)\ -\ E_{32}(-1) E_{31}(-2). 
\end{align*}
Also, let 
  \[ u =\  \tfrac 1 3  a^2 - b, \quad \mbox{and} \quad  v =\  \tfrac 2 9 a^3 -ab -3c.  \] 
Using the same formulas, we may also define elements of $\mathcal{U}(\hat{\mathfrak g}_-)$, which are again denoted $a,b,c,u,v$ and $w$.  The context should make it clear to which algebra the symbols belong.

Now write $B_{p,q,r,s} =  (B^2_{-1})^p (B^3_{-1})^q (B^4_{-1})^r (B^1_{-2})^s$ for $p,q,r,s \in \mathbb{Z}_{\ge 0}$.  It follows from (\ref{eq:def}) and the proof of  Lemma \ref{lem:P1} that $m(B_{p,q,r,s})=q+2r+s$, and one may check that
\begin{equation}\label{eq:gens}
a^p b^q c^r w^s \in \mbox{ Span}_{\mathbb C}\{(B^1_{-1})^{q+2r+s} B_{p,q,r,s}\} 
\end{equation}
for $p,q,r,s \in \mathbb{Z}_{\ge 0}$.

\begin{Lem}\label{lem:P2}
Let $(\mathcal{S}_{\theta^{\vee}-\delta})^{\beta}$ denote the subalgebra of $\mathcal{S}_{\theta^{\vee}-\delta}$ consisting of elements invariant under the adjoint action of $E_{01}$.  Suppose $B$ is a $\beta$-string with no $B^1_{-1}$-factors.  Then
\begin{equation*}
(\mathcal{S}_{\theta^{\vee}-\delta})^{\beta} \cap \mbox{ Span}_{\mathbb{C} } \left((B^1_{-1})^{m(B)} B\right) =
\begin{cases} \mathbb{C} a^p b^q c^r w^s & \text{if $B=B_{p,q,r,s}$, for some $(p,q,r,s) \in (\mathbb{Z}_{\ge 0})^4$}
\\
\quad 0 &\text{otherwise.}
\end{cases}
\end{equation*}
Furthermore, $(\mathcal{S}_{\theta^{\vee}-\delta})^{\beta}$ is a polynomial algebra with generators $a,b,c,w$.
\end{Lem}

\begin{proof}
Let $B = B^{i_1}_{-j_1} \dots B^{i_t}_{-j_t}$ be a $\beta$-string with no $B^1_{-1}$-factors, and suppose 
\[
f \in (\mathcal{S}_{\theta^{\vee}-\delta})^{\beta} \cap \mbox{ Span}_{\mathbb{C} } \left((B^1_{-1})^{m(B)} B\right).
\]
Then by Lemma \ref{lem:P1}, 
\[
f = \sum_{y \in B} d(y) E_{32}(-1)^{m_1(y)} E_{31}(-1)^{m_2(y)} y,
\]
 for some coefficients $d(y)\in \mathbb{C}$.  Since $[E_{01}, E_{31}] = - E_{32}$ in $\mathfrak g$, we may compute:
\begin{align}
E_{01}.f &= -\sum_{y \in B} d(y) m_2(y) E_{32}(-1)^{m_1(y)+1} E_{31}(-1)^{m_2(y)-1} y \label{eq:line1}\\
& +\sum_{y \in B} d(y) m_2(y) E_{32}(-1)^{m_1(y)} E_{31}(-1)^{m_2(y)} \left(\sum_{ \{y'\in B| y' \vdash y\} } c(y,y') y' \right) \nonumber
\end{align}
By (\ref{eq:cover}), if $y'\vdash y$, then $m_1(y) = m_1(y')+1$ and $m_2(y)=m_2(y')-1$.  Rewriting each sum in terms of $y' \in B$, the equation (\ref{eq:line1}) is then equivalent to:
\begin{align*}
E_{01}.f &= -\sum_{y' \in B} d(y') m_2(y') E_{32}(-1)^{m_1(y')+1} E_{31}(-1)^{m_2(y')-1} y'\\ 
& +\sum_{y' \in B\backslash \{y_-\}} \left(\sum_{ \{y \in B| y' \vdash y\} } d(y) c(y,y') \right) E_{32}(-1)^{m_1(y')+1} E_{31}(-1)^{m_2(y')-1}  y' 
\end{align*}
We thus see that $E_{01}.f = 0$ if and only if the following two equations are satisfied:
\begin{equation}\label{eq:first}
-d(y') m_2(y')\ + \sum_{ \{y \in B| y' \vdash y\} } d(y) c(y,y') \ =\ 0 \quad (y'\in B\backslash \{y_-\})
\end{equation}
\begin{equation}\label{eq:second}
d(y_-) m_2(y_-)\ =\ 0
\end{equation}

From Remark \ref{rmk:minimum}, it follows that $m_2(y')>0$ for all $y' \in B\backslash \{y_-\}$.  We may thus solve (\ref{eq:first}) in terms of $d(y')$, and we see that $d(y')$ is linearly dependent on the coefficients $ d(y)$ for which $y' \vdash y$.  This fact may then be used recursively, for all $y'' \preceq y'$, to show that $d(y') = C\ d(y_-)$ for some constant $C\in \mathbb{C}$ which depends only on the integers $m_2(y)$ and $c(y,y')$ $(y,y' \in B)$.  Thus, we have:
\begin{equation}\label{eq:third}
\mbox{dim}_{\mathbb{C}}\left\{(\mathcal{S}_{\theta^{\vee}-\delta})^{\beta} \cap \mbox{ Span}_{\mathbb{C} } \left((B^1_{-1})^{m(B)} B\right)\right\} \le 1.
\end{equation}

Now recall that $a^p b^q c^r w^s \in$ Span$_{\mathbb{C}}( B_{p,q,r,s})$ for $(p,q,r,s) \in (\mathbb{Z}_{\ge 0})^4$.
The elements $a,b,c,w$ belong to $\mathcal{S}_{\theta^{\vee}-\delta}$ since:
\[
wt(a) =\ \theta^{\vee}-\delta, \quad wt(b)=\ 2\theta^{\vee}-2\delta, \quad wt(c) =\ wt(w) =\  3\theta^{\vee}-3\delta.
\]
One may also see that $a,b,c,w \in (\mathcal{S}_{\theta^{\vee}-\delta})^{\beta}$ by checking that: $E_{01}.a = E_{01}.b = E_{01}.c = E_{01}.w = 0$.  It thus follows that the dimension in (\ref{eq:third}) is equal to 1 when $B=B_{p,q,r,s}$ for some $(p,q,r,s) \in (\mathbb{Z}_{\ge 0})^4$.

Remark \ref{rmk:minimum} also shows that $m_2(y_-)=0$ if and only if $B=B_{p,q,r,s}$ for some $(p,q,r,s)\in (\mathbb{Z}_{\ge 0})^4$.  From the second equation (\ref{eq:second}) above, we see that $d(y_-)=0$ whenever $m_2(y_-)>0$.  Hence, if $B \neq B_{p,q,r,s}$ for any $(p,q,r,s)\in (\mathbb{Z}_{\ge 0})^4$, then the dimension in (\ref{eq:third}) is 0.    This shows the first part of the lemma.

From (\ref{eq:gens}) and the fact that $\beta$-strings partition $\mathcal{B}_{PBW}$, it follows that
\[ \{a^p b^q c^r w^s\ |\ p,q,r,s \in \mathbb{Z}_{\ge 0} \} \]
is a linearly independent set.  Thus $\{a,b,c,w\}$ is a set of algebraically independent generators for a polynomial subalgebra of $(\mathcal{S}_{\theta^{\vee}-\delta})^{\beta}$.  Finally, it follows from the first part of the lemma that $a,b,c,w$ in fact generate the entire algebra $(\mathcal{S}_{\theta^{\vee}-\delta})^{\beta}$.
\end{proof}

\begin{Lem}\label{lem:P3}
The subalgebra $(\mathcal{S}_{\theta^{\vee}-\delta})^{\frak{n}_+}$ of $\mathcal{S}_{\theta^{\vee}-\delta}$, consisting of elements invariant under the adjoint action of $\frak{n}_+$, is a polynomial algebra with generators $u,v,w$.
\end{Lem}

\begin{proof}
Since the Lie algebra $\frak{n}_+$ is generated by $E_{10}$ and $E_{01}$, the algebra $(\mathcal{S}_{\theta^{\vee}-\delta})^{\frak{n}_+}$ consists of all elements $f \in \mathcal{S}_{\theta^{\vee}-\delta}$ which are invariant under the adjoint action of both $E_{10}$ and $E_{01}$.  By Lemma \ref{lem:P2}, it is clear from their definition that $u$,$v$,$w$ belong to $(\mathcal{S}_{\theta^{\vee}-\delta})^{\beta}$.  One may check that \[ E_{10}.a = -3 E_{31}(-1),\quad E_{10}.b = -2 E_{31}(-1) a,\]\[\quad E_{10}.c = E_{31}(-1)b,\quad  \mbox{and} \quad E_{10}.w = 0.\]   These equations may be used to show $E_{10}.u= E_{10}.v= E_{10}.w= 0$, and thus  $u,v,w \in (\mathcal{S}_{\theta^{\vee}-\delta})^{\frak{n}_+}$.  Since $u=\tfrac 1 3 a^2 - b$, $v=\tfrac 2 9 a^3 -a b -3c$, and the elements $a,b,c,w$ are algebraically independent in $\mathcal{S}(\hat{\mathfrak g}_-)$, the elements $u,v,w$ are also algebraically independent in $\mathcal{S}(\hat{\mathfrak g}_-)$.

Now suppose $f \in (\mathcal{S}_{\theta^{\vee}-\delta})^{\frak{n}_+}$.  We must show that $f$ is a polynomial in $u,v,w$.  We may assume that $f$ is homogeneous with respect to the $\widehat{Q}$-grading, say $wt(f)= n(\theta^{\vee}-\delta)$.  Since $E_{01}.f=0$, it follows from Lemma \ref{lem:P2} that
\[ f = \sum_{\substack{p,q,r,s \in \mathbb{Z}_{\ge 0}\\ p+2q+3r+3s = n} } C_{p,q,r,s} a^p b^q c^r w^s,\]
for some constants $C_{p,q,r,s}\in \mathbb{C}$.  Since $E_{10}.w = 0$, it is sufficient to consider the case where $C_{p,q,r,s}=0$ whenever $s>0$, and we write $C_{p,q,r}=C_{p,q,r,0}$.  

Next, by assumption $E_{10}.f =0$, and we expand this equation as follows:
\begin{align}
E_{10}.f =&\phantom{LL} \sum_{\substack{ p+2q+3r = n} } C_{p,q,r} E_{10}.( a^p b^q c^r)\notag \\
=&\phantom{LL} \sum_{\substack{ p+2q+3r = n} } C_{p,q,r} \tbinom{p}{1} (-3)E_{31}(-1)a^{p-1} b^q c^r \label{eq:thirdpart} \\
&+ \sum_{\substack{ p+2q+3r = n} } C_{p,q,r} \tbinom{q}{1} (-2)E_{31}(-1)a^{p+1} b^{q-1} c^r\notag \\
&+ \sum_{\substack{ p+2q+3r = n} } C_{p,q,r} \tbinom{q}{1} E_{31}(-1)a^p b^{q+1} c^{r-1}\notag \\
 = &\ \  0. \notag
\end{align}
Now the set
 \begin{equation*}
 \{a^{p'} b^{q'} c^{r'} E_{31}(-1)|\ p',q',r' \in \mathbb{Z}_{\ge 0} \text{ and } p'+2q'+3r'=n-1\}
 \end{equation*}
  is linearly independent. Upon setting the coefficient of each $a^{p'} b^{q'} c^{r'} E_{31}(-1)$ equal zero in (\ref{eq:thirdpart}), we obtain the equations
\begin{equation}\label{eq:coeff1}
 -3(p'+1)C_{p'+1,q',r'} -2(q'+1)C_{p'-1,q'+1,r'}+ (r'+1)C_{p',q'-1,r'+1}\ =\ 0
 \end{equation}
 for $p'+2q'+3r' = n-1$, where we set $C_{p',q',r'} = 0$ whenever $p'$,$q'$, or $r'$ is negative.

After setting $p=p'+1, q=q', r=r'$ in 
(\ref{eq:coeff1}), we may rewrite the equations as follows:
\begin{equation}\label{eq:coeff2} 3p\ C_{p,q,r} =  2p\ C_{p-2,q+1,r}- (r+1)C_{p-1,q-1,r+1} \quad \quad \quad (p+2q+3r = n).
\end{equation}
We may then use
(\ref{eq:coeff2}) recursively to write each of the variables $C_{p,q,r}$, with $p>0$ and $q,r \ge 0$, as a linear combination of the variables $C_{0,q,r}$, with $2q + 3r =n$.  Hence the number of independent solutions to this system is at most the number of ways, $d(n)$, to partition $n$ into $n = 2q +3r$: \begin{equation}\label{eq:parts}d(n) = \left \{ \begin{array}{ll} \left \lfloor \tfrac{n+6}{6} \right \rfloor & \text{ if $n$ is even} \\
\left \lfloor \frac{n+3}{6} \right \rfloor & \text{ if $n$ is odd,}
\end{array}\right. \end{equation}
where $\lfloor \cdot \rfloor$ is the greatest integer function.  On the other hand, there are $d(n)$ independent solutions, $f$, to the system (\ref{eq:coeff2}) given by the linearly independent set\begin{equation}\label{eq:set} \{f = u^q v^r | 2q + 3r =n \}, \end{equation}
since $wt(u^q v^r) = n(\theta^{\vee} - \delta)$ whenever $2q + 3r = n$.  Hence, (\ref{eq:set}) gives all $d(n)$ solutions to the system, and any solution $f$ to (\ref{eq:thirdpart}) must be a polynomial in $u, v$, which is what we needed to show. 
\end{proof}

We now give the final lemma of the section.  Recall that $u,v,w$ may also be regarded as elements of $\mathcal{U}(\hat{\mathfrak g}_-)$.
\begin{Lem} \label{lem:P3:universal}
The subalgebra $(\mathcal{U}_{\theta^{\vee}-\delta})^{\mathfrak n_+}$ of  $\mathcal{U}_{\theta^{\vee}-\delta}$, consisting of elements invariant under the adjoint action of $\frak{n}_+$, is generated by 1 along with the elements $u,v,w$.
\end{Lem}

\begin{proof}
Recall the $\widehat{Q}$-gradings of $\mathcal{S}(\hat{\mathfrak g}_-)$ and $\mathcal{U}(\hat{\mathfrak g}_-)$.  We have
\[
(\mathcal{S}_{\theta^{\vee}-\delta})^{\mathfrak n_+} = \mathcal{S}^0 \oplus \mathcal{S}^1 \oplus \mathcal{S}^2 \oplus \cdots \qquad (\mbox{resp. } (\mathcal{U}_{\theta^{\vee}-\delta})^{\mathfrak n_+} = \mathcal{U}^0 \oplus \mathcal{U}^1 \oplus \mathcal{U}^2 \oplus \cdots),
\]
where we set
\[
\mathcal{S}^n = (\mathcal{S}_{\theta^{\vee}-\delta})^{\mathfrak n_+} \cap \mathcal{S}(\hat{\mathfrak g}_-)_{(n\theta^{\vee}-n\delta)} \qquad
\left(\mbox{resp. }
\mathcal{U}^n = (\mathcal{U}_{\theta^{\vee}-\delta})^{\mathfrak n_+} \cap \mathcal{U}(\hat{\mathfrak g}_-)_{(n\theta^{\vee}-n\delta)}\right)
\]
for $n \in \mathbb{Z}_{\ge 0}$.
Since the map $\omega: \mathcal{S}(\hat{\mathfrak g}_-) \rightarrow \mathcal{U}(\hat{\mathfrak g}_-)$ is  a $\mathfrak g$-module isomorphism, it follows from (\ref{eq:grading1}) that both
\begin{equation}\label{eq:mapping1}
\omega \left( (\mathcal{S}_{\theta^{\vee}-\delta})^{\mathfrak n_+}\right) = (\mathcal{U}_{\theta^{\vee}-\delta})^{\mathfrak n_+} \qquad \mbox{and} \qquad \omega(\mathcal{S}^n) = \mathcal{U}^n\ (n \in \mathbb{Z}_{\ge 0}).
\end{equation}
Using (\ref{eq:isom1}), we may check that
\begin{equation}\label{eq:mapping2}
\omega(u)=u, \quad \omega(v)=v-3w, \ \ \mbox{ and } \ \ \omega(w)=w.
\end{equation}
From Lemma \ref{lem:P3}, we have
\[
\mathcal{S}^n = \mbox{Span}_{\mathbb{C} }\{u^p v^q w^r\ |\ p,q,r \in \mathbb{Z}_{\ge 0}; 2p + 3q + 3r = n \},
\]
and it follows from (\ref{eq:mapping2}) that
\[
 \mathcal{U}^n \supseteq \mbox{Span}_{\mathbb{C} }\{u^p v^q w^r\ |\ p,q,r \in \mathbb{Z}_{\ge 0}; 2p + 3q + 3r = n \}.
\]
From (\ref{eq:mapping1}), it is thus sufficient to show that the subset $\{u^p v^q w^r\ |\ p,q,r \in \mathbb{Z}_{\ge 0}\} \subseteq (\mathcal{U}_{\theta^{\vee}-\delta})^{\mathfrak n_+}$ is linearly independent.

Consider $\mathcal{B}_{PBW}$ as a basis for $\mathcal{U}(\hat{\mathfrak g}_-)$.  For each $y \in \mathcal{B}_{PBW}$, let 
$ \pi_y: \mathcal{U}(\hat{\mathfrak g}_-) \rightarrow \mathbb{C} y$
denote the projection which maps $y$ to itself and all other basis elements to zero.
For each $(p,q,r) \in (\mathbb{Z}_{\ge 0})^3$, we may find an element $y(p,q,r) \in \mathcal{B}_{PBW}$ such that the following holds:  if $p',q',r' \in \mathbb{Z}_{\ge 0}$ and $r' \ge r$, then
\begin{equation}\label{eq:mapping5}
\pi_{y(p,q,r)}(u^{p'}v^{q'}w^{r'}) \neq 0, \quad \text{if and only if} \quad \text{ $p'=p$, $q'=q$, and $r'=r$.}
\end{equation}
  For example, if we let
\[
y(p,q,r) = E_{31}(-1)^{p+2	q+r} E_{11}(-1)^{p} E_{01}(-1)^{q} E_{32}(-2)^{r} \in \mathcal{B}_{PBW}
\]
for each $(p,q,r) \in (\mathbb{Z}_{\ge 0})^3$, then (\ref{eq:mapping5}) is satisfied.

Thus, let us fix elements $y(p,q,r)\in \mathcal{B}_{p,q,r}$ which satisfy (\ref{eq:mapping5}), and write $\pi_{p,q,r}=\pi_{y(p,q,r)}$ for each $(p,q,r) \in (\mathbb{Z}_{\ge 0})^3$.  Then suppose we have  an equation
\begin{equation}\label{eq:zerosum}
\sum_{ p',q',r' \in \mathbb{Z}_{\ge 0}  } C_{p',q',r'}\ u^{p'} v^{q'} w^{r'}\ =\ 0,
\end{equation}
where $C_{p',q',r'} \in \mathbb{C}$ are constants of which only finitely many are nonzero.  After applying the projection $\pi_{p,q,0}$ to both sides of the equation (\ref{eq:zerosum}), we obtain the equation $C_{p,q,0}u^p v^q = 0$, whenever $p,q \in \mathbb{Z}_{\ge 0}$.  Hence $C_{p,q,0} = 0$ for all $p,q \in \mathbb{Z}_{\ge 0}$.  We may then proceed in a similar fashion to show by induction that $C_{p,q,r} = 0$ for all $p,q,r \in \mathbb{Z}_{\ge 0}$.  Therefore $\{ u^p v^q w^r\ |\ p,q,r \in \mathbb{Z}_{\ge 0} \}$ is a linearly independent subset of $(\mathcal{U}_{\theta^{\vee}-\delta})^{\mathfrak n_+}$.
\end{proof}

\bigskip

\section{Formula for the singular vectors}
Let $k = -2 + m + \tfrac i 3$ for some $m \in \mathbb{Z}_{\ge 0}$ and $i\in \{1,2\}$, and let $n=3k+7$.  It follows from Lemma \ref{lem:exist} that there exists a singular vector in $N(k,0)$ which generates the maximal submodule $J(k,0)$ and satisfies properties $(P1)-(P4)$ of the lemma.  In this section, we will describe the set of elements in $N(k,0)$ which satisfy all four of these properties, and this description allows us to give a formula for the singular vectors.

Suppose first that $v_0 \in N(k,0)$ satisfies the three properties $(P1)-(P3)$.  I.e., suppose that
\begin{equation*}
v_0 \in \left( N(k,0)_{ (k\Lambda_0+ n\theta^{\vee}-n\delta) } \right)^{\mathfrak n_+}.
\end{equation*}
Since the map $\phi: \mathcal{U}(\hat{\mathfrak g}_-) \rightarrow N(k,0)$ is a $\mathfrak g$-module isomorphism, it follows from (\ref{eq:grading2}) that
\[ v_0 = f.1, \quad \mbox{ for some } f \in \left(\mathcal{U}(\hat{\mathfrak g}_-)_{(n\theta^{\vee}-n\delta)}\right)^{\mathfrak n_+}.
\]
From Lemma \ref{lem:P3:universal}, we thus have
\begin{equation}\label{eq:sing0}
v_0 = \sum_{p,q,r \in\mathbb{Z}_{\ge 0} } C_{p,q,r}\ u^p v^q w^r. \mathbf{1},
\end{equation}
for some constants $C_{p,q,r} \in \mathbb{C}$, such that $C_{p,q,r} \neq 0$ only if $2p+3q+3r=n$.  

Now if $v_0 \in N(k,0)_{(k\Lambda_0+n\theta^{\vee}-n\delta)}$ is a singular vector, then it must satisfy all four of the properties $(P1)-(P4)$.  In this case, we are able to explicitly determine the coefficients appearing in (\ref{eq:sing0}).  For the purpose of classification using Zhu's theory, however, it is only necessary to compute a certain subset of these coefficients.  Recall from Proposition 1.5 of \cite{AL} (see also \cite{FZ}) that we have an isomorphism $F: A(N(k,0)) \rightarrow \mathcal{U}(\frak g)$, given by 
\begin{equation} \label{eq:isomF}
 F( [a_1(-n_1-1) \cdots a_m(-n_m-1) \mathbf{1}]) = (-1)^{n_1+ \cdots +n_m} a_1 \cdots a_m,
 \end{equation}
  for  $a_1, ... , a_m \in \frak g$ and  $n_1, ... , n_m \in \mathbb Z_{+}$.  It is then clear from the definition of $w$ that $F([f w.\mathbf 1]) = 0$ for all $f \in \mathcal{U}(\hat{\mathfrak g }_-)$.  Hence, it is sufficient to only compute the formula for a singular vector modulo the subspace $\mathcal{U}(\hat{\mathfrak g}_-)w.\mathbf{1} \subseteq N(k,0)$.  I.e., we need only to compute the coefficients of the form $C_{p,q,0}$ appearing in (\ref{eq:sing0}).

We now give a formula for singular vectors modulo $\mathcal{U}(\hat{\mathfrak g}_-)w.\mathbf{1}$.  In the proof we make use of the notation and lemmas provided in Appendix A.
\begin{Prop}\label{prop:sing}
Let $k = -2 + m + \frac{i}{3}$ for some $m \in \mathbb{Z}_{\ge 0}$ and $i \in \{1,2\}$, and let $n=3k+7$.  Suppose that $v_0 \in N(k,0)_{(k\Lambda_0+n\theta^{\vee}-n\delta)}$
is a singular vector.  
If $n$ is even, then
\[v_0 \equiv \quad \sum _{j=0}^{d(n)-1} b_j u^{n/2-3j	} v^{2j}.\mathbf{1} \quad \mbox{ mod } \mathcal{U}(\hat{\mathfrak g}_-)w.\mathbf{1},\]
where $d(n)$ is the same as in (\ref{eq:parts}), $b_0 \in \mathbb{C}\backslash\{0\}$, and 
\begin{equation}\label{eq:sing1}
b_j = \frac{2^j \binom{n/2}{2} \binom{n/2-3}{2}\cdots\binom{n/2-3j+3}{2}}{3^j (2j)!}b_0 \quad (j=1, \dots, d(n)-1).
\end{equation}
If $n$ is odd, then
\[v_0 \equiv \quad \sum _{j=0}^{d(n)-1} b_j u^{(n-3)/2-3j} v^{1+2j}.\mathbf{1} \quad \mbox{ mod } \mathcal{U}(\hat{\mathfrak g}_-)w.\mathbf{1},\]
where $b_0 \in \mathbb{C}\backslash\{0\}$ and 
\begin{equation}\label{eq:sing2}
b_j = \frac{2^j \binom{(n-3)/2}{2} \binom{(n-3)/2-3}{2}\cdots\binom{(n-3)/2-3j+3}{2}}{3^j (2j+1)!}b_0 \quad (j=1, \dots, d(n)-1).
\end{equation}
\end{Prop}

\begin{proof}  The definition of a singular vector guarantees that $\hat{\mathfrak n}_+.v_0 = 0$.  Since $\hat{\mathfrak n}_+$ is generated by the elements $E_{10}(0), E_{01}(0), F_{32}(1) \in \hat{\mathfrak g}$, we must have $E_{10}(0).v_0 = E_{01}(0).v_0 = F_{32}(1).v_0 =0$.  Since $E_{10}(0)$ and $E_{01}(0)$ together generate $\mathfrak n_+$, it follows from (\ref{eq:sing0}) that
\[
v_0 = \sum_{p,q,r \in\mathbb{Z}_{\ge 0} } C_{p,q,r}\ u^p v^q w^r. \mathbf{1},
\]
for some constants $C_{p,q,r} \in \mathbb{C}$.  
Rewriting the equation $F_{32}(1).v_0=0$, we then have
\begin{equation}\label{eq:sing3}
 \sum_{p,q,r } C_{p,q,r}\ F_{32}(1).\left( u^p v^q w^r. \mathbf{1} \right) = 0.
\end{equation}

Next, for each $(p',q') \in (\mathbb{Z}_{\ge 0})^2$, let us apply the projection $\pi_{p',q'}$ from Appendix A to both sides of the equation (\ref{eq:sing3}).  By Lemma \ref{lem:app:a}, we obtain the equation
\begin{align}
\pi_{p',q'}(F_{32}(1).v_0) =\ &-2\ C_{p'+2,q',0}\ \ \ (-1)^{p'+q'}3^{q'} \tbinom{p'+2}{2} y(p',q').\mathbf{1} \nonumber  \\
& +6\ C_{p'-1,q'+2,0} (-1)^{p'+q'}3^{q'} \tbinom{q'+2}{2} y(p',q').\mathbf{1} \label{eq:sing4} \\
 =\ &\ \ 0. \nonumber 
\end{align}
Rewriting (\ref{eq:sing4}) we thus have
\begin{equation*}
(-1)^{p'+q'}3^{q'} \left( -2\ C_{p'+2,q',0} \  \tbinom{p'+2}{2} +6\ C_{p'-1,q'+2,0} \tbinom{q'+2}{2} \right) y(p',q').\mathbf{1}\ =\ 0.
\end{equation*} 
Hence, for each $(p',q')\in (\mathbb{Z}_{\ge 0})^2$, we have
\begin{equation}\label{eq:sing5}
 C_{p'+2,q',0} \  \tbinom{p'+2}{2}\ =\ 3\ C_{p'-1,q'+2,0} \tbinom{q'+2}{2}.
\end{equation}
Substituting $p=p'+2$ and $q=q'$ in (\ref{eq:sing5}), we obtain the recurrence relation
\begin{equation}\label{eq:sing6}
 C_{p,q,0} \  \tbinom{p}{2}\ =\ 3\ C_{p-3,q+2,0} \tbinom{q+2}{2}
\end{equation}
for the coefficients of $v_0$.  

Since $wt(v_0) = k\Lambda_0 + n \theta^{\vee}-n\delta$, it follows that $wt(u^p v^q w^r) = n \theta^{\vee} - n\delta$ whenever $C_{p,q,r} \neq 0$.  Notice also that $v_0 \equiv \sum_{p,q} C_{p,q,0} u^p v^q$, mod $\mathcal{U}(\mathfrak g_-)w. \mathbf{1}$.  As in the proof of Lemma \ref{lem:P3}, we see that there are $d(n)$ distinct monomials of the form $u^p v^q$ such that $2p + 3q = n$.  We may thus write
\begin{equation}\label{eq:sing7}
v_0 \equiv 
\begin{cases}
\displaystyle \sum_{j=0}^{d(n)-1} b_j\ u^{n/2-3j} v^{2j} . \mathbf{1} & \text{ if $n$ is even}
\bigskip\\
\displaystyle \sum_{j=0}^{d(n)-1} b_j\ u^{(n-3)/2-3j} v^{2j+1} . \mathbf{1} & \text{ if $n$ is odd},
\end{cases}
\end{equation}
mod $\mathcal{U}(\hat{\mathfrak g}_-)w. \mathbf{1}$, where $b_j = C_{n/2-3j,2j,0}$ (resp. $b_j= C_{(n-3)/2-3j, 2j+1, 0}$) if $n$ is even (resp. odd), for $j=0, \dots, d(n)-1$ .

Rewriting the recurrence relation (\ref{eq:sing6}) in terms of the coefficients $b_j$, we obtain:  
\begin{equation}\label{eq:sing8}
b_{j+1} = 
\begin{cases}
 \displaystyle \frac{\binom {n/2-3j+3}{2}}{3 \binom {2j}{2}} b_j & \text{ if $n$ is even}
\bigskip \\ 
 \displaystyle \frac{\binom{(n-3)/2-3j+3}{2} }{3\binom{2j+1}{2}} b_j & \text{ if $n$ is odd},
\end{cases} 
\end{equation}
 for $j = 1,\dots,d(n)$.  Since $v_0 \neq 0$ by the definition of a singular vector, we must have $b_0 \in \mathbb{C}\backslash \{0\}$.  We may then obtain the desired formulas (\ref{eq:sing1}) and (\ref{eq:sing2}) for  coefficients $b_j$ by using (\ref{eq:sing8}) recursively for $j=1, \dots, d(n)-1$.
\end{proof}

\begin{Thm}\label{thm:sing}
Let $k = -2 + m + \frac{i}{3}$ for some $m \in \mathbb{Z}_{\ge 0}$ and $i \in \{1,2\}$, and let $n=3k+7$.  Then there exists a singular vector $v_k \in N(k,0)$, such that $v_k$ generates the maximal ideal $J(k,0)$ of $N(k,0)$ and
\begin{equation*}
v_k \equiv
\begin{cases}
\displaystyle \sum _{j=0}^{d(n)-1} b_j u^{n/2-3j	} v^{2j}.\mathbf{1} & \text{if $n$ is even} 
\medskip\\
\displaystyle \sum _{j=0}^{d(n)-1} b_j u^{(n-3)/2-3j} v^{1+2j}.\mathbf{1} &\text{if $n$ is odd,}
\end{cases} 
\end{equation*}
mod $\mathcal{U}(\mathfrak g_-)w.\mathbf{1}$, where $b_0 = 1$ and $b_j\ (j=1,\dots, d(n)-1)$ are given by formulas (\ref{eq:sing1}) and (\ref{eq:sing2}),respectively.
\end{Thm}

\begin{proof}
The existence of a singular vector $v_k \in N(k,0)$, with the properties that $wt(v_k) = k\Lambda_0 +n\theta^{\vee} - n\delta$ and $v_k$ generates the maximal ideal $J(k,0)$, follows from Lemma \ref{lem:exist}.  
Then $v_k$ satisfies the conditions of Propostion \ref{prop:sing}, and the theorem follows by setting $b_0=1$ in formulas (\ref{eq:sing1}) and (\ref{eq:sing2}), respectively.
\end{proof}

Now we consider the image of the singular vector $v_k$ under Zhu's map \[ [\cdot ]: N(k,0)\rightarrow A(N(k,0)). \]  In what follows, we will identify $N(k,0)$ with $\mathcal{U}(\hat{\frak{g}}_-)$ via the map $\phi$ from Section 1.  Similarly, we identify $A(N(k,0))$ with $\mathcal{U}(\mathfrak g)$ using the map $F$, given in (\ref{eq:isomF}).  As in \cite{AL}, we then have an induced map $[\cdot] :\mathcal{U}(\hat{\frak{g}}_-) \rightarrow \mathcal{U}(\frak{g})$.  We compute:
\begin{equation*}
[a] = E_{21}, \quad
[b] = E_{31} E_{11} - E_{32} E_{10} ,\quad
[c] =E^2_{31}E_{01} - E_{32}E_{31}H_{01} - E^2_{32}F_{01}.
\end{equation*}
We also have:
\begin{equation}\label{eq:zhu0}
[u] = \tfrac 1 3 [a]^2 - [b], \qquad [v] = \tfrac 2 9 [a]^3 -[a][b] - 3[c].
\end{equation}

Since $[w]=0$, we have the following from Proposition \ref{prop:sing}:
\begin{equation}\label{eq:zhu1}
[v_k] = \quad \sum _{i=0}^{d(n)-1} b_i [u]^{n/2-3i	} [v]^{2i},
\end{equation}
if $n$ is even, and 
\begin{equation}\label{eq:zhu2}
[v_k] = \quad \sum _{i=0}^{d(n)-1} b_i [u]^{(n-3)/2-3i} [v]^{1+2i},
\end{equation}
if $n$ is odd, where $n = 3k+7$ as before and the coefficients $b_i$ are given in (\ref{eq:sing1}) and (\ref{eq:sing2}), respectively.

The following theorem is now a consequence of Proposition 1.6 of \cite{AL} and Theorem \ref{thm:sing}.

\begin{Thm} \label{thm:zhu:image}
Let $k$ be an admissible one-third integer level, and let $v_k \in N(k,0)$ be a singular vector of the form given in Theorem \ref{thm:sing}.  Then the associative algebra $A(L(k,0))$ is isomorphic to $\mathcal{U}(\frak{g})/I_{k}$,  where $ I_{k}$ is the two-sided ideal of $\mathcal{U}(\frak{g})$ generated by $[v_k]$.
\end{Thm}

\bigskip


 \section{Irreducible modules}

Denote by $_L$ the adjoint action of $\mathcal{U}(\frak{g})$ on itself defined by $X_Lf = { [ X, f]}$ for $ X \in \frak{g}$ and $f \in \mathcal{U}(\frak{g})$.  Let $R(k)$ be the $\mathcal{U}(\frak{g})$-submodule of $\mathcal{U}(\frak{g})$ generated by $[v_{k}]$, where $[v_k]$ is given in (\ref{eq:zhu1}) and (\ref{eq:zhu2}), respectively.  Then $R(k)$ is an irreducible finite-dimensional $\mathcal{U}(\frak{g})$-module isomorphic to $V(n\theta^{\vee})$, where $n=3k+7$.  Let $R(k)_0$ be the zero-weight subspace of $R(k)$.
 
  \begin{Prop} \cite{A1,AM}
 Let $V(\mu)$ be an irreducible highest weight $\mathcal{U}(\mathfrak{g})$-module with highest weight vector $v_\mu$ for $\mu \in \mathfrak{h}^*$.  Then the following statements are equivalent:
\begin{enumerate}
\item  $V(\mu)$ is an $A(L(k, 0))$-module,
\item   $R(k) \cdot V(\mu) = 0$,
\item   $R(k)_0 \cdot v_\mu = 0$. 
\end{enumerate}
 \end{Prop}

 Let $r\in R(k)_0$.  Then there exists a unique polynomial $p_r \in \mathcal{S}(\frak{h})$, where $\mathcal{S}(\frak{h})$ is the symmetric algebra of $\frak h$,  such that \[r \cdot v_\mu = p_r(\mu) v_\mu. \]  Set $\mathcal{P}(k)_0 = \{p_r \, | \, r \in \mathcal {R}(k)_0 \}$.  We thus have:
 
 \begin{Cor}  \label{cor:biject} There is a bijective correspondence between \begin{enumerate} \item the set of irreducible $A(L(k, 0) )$-modules $V(\mu)$ from the category $\mathcal{O}$, and \item the set of weights $\mu \in \frak{h}^*$ such that $p(\mu)=0$ for all $p \in \mathcal{P}(k)_0$. \end{enumerate}
 \end{Cor}

In the next Lemma, we make use of the notation and results from Appendix B.  In particular, recall the PBW gradation: $\mathcal{S}(\mathfrak h) = \bigoplus_{j=0}^{\infty} \mathcal{S}(\mathfrak h)^j $.  

\begin{Lem}\label{lem:polynomials}
  Suppose $k=-2+m+\tfrac{i}{3}$ is an admissible one-third integer level and $n=3k+7$.
Let $p_1(H), p_2(H)$ be the unique polynomials in $\mathcal{P}(k)_0$ which satisfy
\begin{align*}
& p_1(H)  \equiv \ \frac{1}{(n!)^2} (E_{10}^n F_{31}^n)_L [v_k],\\
& p_2(H)  \equiv \ \frac{1}{(n!)^2} (E_{11}^n F_{32}^n)_L [v_k] \quad (\mbox{mod }\mathcal{U}(\mathfrak{g})\mathfrak{n}_+),
\end{align*}
respectively.  Then we have:
\begin{enumerate}
\item  $p_1(H) = C_1 \cdot H_{10}(H_{10}-1)\cdots (H_{10}-n+1)$ 
\item $p_2(H) \equiv C_2 \cdot H_{11}(H_{11}-1)\cdots (H_{11}-n+1),\quad \text{ mod } \bigoplus_{j=0}^{n-1}\mathcal{S}(\mathfrak h)^{j}$
\end{enumerate}
for some constants $C_1,C_2 \in \mathbb{C}\backslash \{0\}$. 
\end{Lem}

\begin{proof}
First notice from (\ref{eq:zhu0}), (\ref{eq:zhu1}), and (\ref{eq:zhu2}) that $[v_k]$ is a sum of the form: \[[v_k]\ =\ \sum_{(p,q,r) \in (\mathbb{Z}_{\ge 0})^3 }C_{p,q,r}\ [a]^p [b]^q [c]^r\ =\ C_{n,0,0}\ [a]^n +  \sum_{ (p,q,r) \neq (n,0,0) }C_{p,q,r}\ [a]^p [b]^q [c]^r, \] for some constants $C_{p,q,r} \in \mathbb{C}$, such that $C_{n,0,0} \neq 0$ and $C_{p,q,r} = 0$ unless $p+2q+3r = n$.  

From Remark \ref{rmk:app:b}, we have
\begin{equation*}
p_1(H) = \pi_0 \left( \tfrac{1}{n!} (E_{10}^nF_{31}^n)_L[v_k] \right)
\end{equation*}
and
\begin{equation*}
p_2(H) = \pi_0 \left( \tfrac{1}{n!} (E_{11}^nF_{32}^n)_L[v_k] \right).
\end{equation*}
Then from Lemma \ref{lem:app:b3} and the fact that $\pi_0: \mathcal{U}(\mathfrak g) \rightarrow \mathcal{S}(\mathfrak h)$ is an algebra homomorphism, we have
\begin{align}
p_1(H)\quad \ =\ &\ \ C_{n,0,0} H_{10}\cdots (H_{10}-n+1) \label{eq:poly1}\\
 + \sum_{(p,q,r)\neq(n,0,0)}&   \frac{1}{((n-p)!)^2} (H_{10}-n+1) \cdots (H_{10}-n+p)\ \pi_0\left( (E_{10}^{n-p}F_{31}^{n-p})_L ([b]^q[c]^r)\right) \nonumber
\end{align}
and
\begin{align}
p_2(H)\quad \ =\ &\ C_{n,0,0} H_{11}\cdots (H_{11}-n+1) \label{eq:poly2}\\
 + \sum_{(p,q,r)\neq(n,0,0)}&  \frac{1}{((n-p)!)^2} (H_{11}-n+1) \cdots (H_{11}-n+p)\ \pi_0 \left( (E_{11}^{n-p}F_{32}^{n-p})_L ([b]^q[c]^r)\right) \nonumber
\end{align}

Then by setting $m=n-p$ in Lemma \ref{lem:app:b}, we notice that
\begin{equation*}
\pi_0\left( (E_{10}^{n-p}F_{31}^{n-p})_L ([b]^q[c]^r)\right) = 0
\end{equation*}
in (\ref{eq:poly1}), whenever $C_{p,q,r}\neq 0$ and  $(p,q,r) \neq (n,0,0)$.  Again since $\pi_0$ is an algebra homomorphism, it follows that 
\begin{equation*}
p_1(H) = C_{n,0,0} H_{10} \cdots (H_{10}-n+1).
\end{equation*}
This proves the first part of the lemma since $C_{n,0,0} \neq 0$.

Similarly, using the definition of the PBW grading on $\mathcal{S}(\mathfrak h)$, we may again apply Lemma \ref{lem:app:b} to see that in the case of $p_2(H)$ we have:
\begin{equation*}
\mathrm{deg}\ \pi_0 \left( (E_{11}^{n-p}F_{32}^{n-p})_L ([b]^q[c]^r)\right) < n-1,
\end{equation*}
whenever $C_{p,q,r}\neq 0$ and  $(p,q,r) \neq (n,0,0)$.  Thus we see that
\begin{equation*}
p_2(H) \equiv C_{n,0,0} \cdot H_{11}(H_{11}-1)\cdots (H_{11}-n+1),\quad \mbox{ mod } \bigoplus_{j=0}^{n-1}\mathcal{S}(\mathfrak h)^{j},
\end{equation*}
which completes the proof.
\end{proof}

\begin{Prop}
Let $k$ be an admissible one-third integer level.  Then there are finitely many irreducible $A(L(k,0))$-modules from the category $\mathcal{O}$.
\end{Prop}

\begin{proof}
Since $p(H)= H_{10}(H_{10}-1)\cdots(H_{10}-n+1) \in \mathcal{P}_0$, any weight $\mu$ such that $p(\mu) = 0$ must satisfy $\mu_{10} \in \{0, 1,2,\dots,n-1\}$.  By substituting any of these values for $H_{10}$ into the polynomial $p_2(H)$ we obtain a polynomial in a single variable, $H_{01}$.  From the second part of Lemma \ref{lem:polynomials}, it follows that the resulting polynomial of $H_{01}$ has degree $n$.  Hence, it is nonzero and has only a finite number solutions.  By Corollary \ref{cor:biject}, it follows then that the number of irreducible $A(N(k,0))$-modules is finite.
\end{proof}

From the theory of Zhu's associative algebra, we then have the following.

\begin{Thm}\label{thm:main}
Let $k$ be an admissible one-third integer level.  Then there are finitely many irreducible weak $L(k,0)$-modules from the category $\mathcal{O}$, for each admissible one-third integer level $k$.
\end{Thm}

\bigskip

\appendix

\section{}
In this appendix, we provide some lemmas which are necessary to prove Proposition \ref{prop:sing}.  For the purpose of simplifying computations, it will be convenient to introduce the following subspace of $N(k,0)$.  Define the $C_2$-subspace of $N(k,0)$ to be 
\begin{equation*}
C_2(N(k,0)) = \mbox{Span}_{\mathbb{C}}\{ a(-2).b\ |\ a \in \frak{g}, b\in N(k,0)\}.
\end{equation*}  
This coincides with the subspace introduced by Zhu in \cite{Z}.  

It is not difficult to check that the $C_2$-subspace satisfies the following properties.
\begin{equation}\label{eq:c2space:a1}
 \mathcal{U}(\hat{\mathfrak g}_-). C_2(N(k,0)) \subseteq C_2(N(k,0))
\end{equation}
\begin{equation}\label{eq:c2space:a2}
 w.\mathbf{1} \in C_2(N(k,0))
\end{equation}
\begin{equation}\label{eq:c2space:a3}
 f g .\mathbf{1} \equiv g f. \mathbf{1}, \mbox{ mod } C_2(N(k,0)) \quad \left(f,g \in  \mathcal{U}(\hat{\mathfrak g}_-)\right)
\end{equation}

The following lemma is also straightforward to verify.
\begin{Lem}\label{lem:c2space}
Let $f_1, \dots, f_t \in \mathcal{U}(\hat{\mathfrak g}_-)$ and $X(1) = X\otimes t \in \hat{\mathfrak g}$, for some $X \in \mathfrak g$.  Then
\begin{align*}
X(1). (f_1 \cdots f_t. \mathbf{1}) &\ \equiv\ \sum_{1 \le i \le t} f_1 \cdots \hat{f_i} \cdots f_t [X(1), f_i]. \mathbf{1}\\
& + \sum_{1\le i< j \le t} f_1 \cdots \hat{f_i} \cdots \hat{f_j} \cdots f_t \left[ [X(1), f_i], f_j\right]. \mathbf{1},
\end{align*}
mod $C_2(N(k,0))$.
\end{Lem}

Next recall that we have the following equations.
\begin{Lem} \label{lem:app:a1} \cite{AL}
The following hold in $\mathcal{U}(\hat{\mathfrak g} )$:
\begin{align*}  
  [u, F_{32}(1)]& =\ - \left (K+ \tfrac 5 3 \right ) E_{10}(-1) - E_{31}(-1)F_{21}(0)  - \tfrac 2 3 E_{21}(-1)F_{11}(0) \\ & \quad\  
+ E_{11}(-1)F_{01}(0) + E_{10}(-1)H_{32}(0) ,\\
 \  [v,  F_{32}(1)] &=\ - E_{32}(-1)E_{10}(-1)F_{11}(0) - 3 E_{32}(-1)F_{01}(-1) \\ & \quad\  + \tfrac 4 3 E_{31}(-2) + E_{31}(-1)E_{11}(-1)F_{11}(0) -E_{31}(-1)H_{11}(-1) \\ & \quad\  - \tfrac 2 3 a^2 F_{11}(0) - \tfrac 1 3 a E_{10}(-1) -a [b, F_{32}(1)] -3 [ c, F_{32}(1)] , \\
 [w,  F_{32}(1)] &=\  - E_{32}(-2)F_{01}(0) +E_{32}(-1)F_{01}(-1) \\ & \quad\  -  E_{31}(-2)H_{32}(0)+ E_{31}(-1)H_{32}(-1) +K E_{31}(-2).
\end{align*} 
\end{Lem}

The equations of Lemma \ref{lem:app:a1} may be used to verify the below equivalences in $N(k,0)$ modulo the $C_2$-subspace.
\begin{Lem}\label{lem:app:a2}  The following congruence relations hold in $N(k,0)$, modulo $C_2(N(k,0))$.
\begin{align*}
&[F_{32}(1),u].\mathbf{1} &\equiv& \left (K+ \tfrac 5 3 \right ) E_{10}(-1).\mathbf{1}, \\ \ 
&[F_{32}(1),v].\mathbf{1} &\equiv& \  \{3 E_{32}(-1)F_{01}(-1) + E_{31}(-1)H_{11}(-1) + \tfrac 1 3 a E_{10}(-1) \\& & &  + (K+1) a E_{10}(-1) +6(K+1) E_{32}(-1)F_{01}(-1)\\
& & & +3(K+1) E_{31}(-1)H_{01}(-1)\}.\mathbf{1} , \\ \
&[F_{32}(1),w].\mathbf{1} &\equiv& \ \{-E_{32}(-1)F_{01}(-1) - E_{31}(-1)H_{32}(-1)\}.\mathbf{1} ,  \\ \
&[[F_{32}(1),u],u].\mathbf{1} &\equiv& \ \{ - \frac 2 3 aE_{31}(-1)H_{10}(-1) - 2 E_{31}(-1)^2F_{10}(-1) - 2 E_{32}(-1)E_{31}(-1)F_{11}(-1)\\& & & + 2 aE_{32}(-1)F_{01}(-1) - 2 uE_{10}(-1) \} .\mathbf{1} , \\ \
&[[F_{32}(1),u],v].\mathbf{1}  &\equiv& \ \{ -\frac 2 3 E_{31}(-1)H_{11}(-1) + b E_{31}(-1) H_{21}(-1)
+ \frac 2 9 a^3 E_{10}(-1)\\& & &- \frac 4 3 ab E_{10}(-1) -3 a^2 E_{32}(-1)F_{01}(-1) - 3 v E_{10}(-1)\\
& & &-6 E_{31}(-1)E_{10}(-1)E_{01}(-1) + 6 E_{31}(-1)E_{11}(-1)H_{01}(-1)\\& & & + 6 E_{32}(-1)E_{10}(-1)H_{01}(-1)+6E_{32}(-1)E_{11}(-1)F_{01}(-1)\} .\mathbf{1} ,
\end{align*}
\begin{align*}
&[[F_{32}(1),v],v].\mathbf{1} &\equiv& \ \{(\tfrac 2 3 a^2-2b)(\tfrac 1 3 a^2 - 2b)E_{10}(-1) - 3a E_{31}^2(-1)E_{10}(-1) - 3 v E_{10}(-1)\\
& & & +2 u a E_{31}(-1)H_{10}(-1)-6v E_{31}(-1)E_{31}(-1)H_{10}(-1) - 9 v E_{31}(-1)H_{01}(-1)\\
& & & + 18 c E_{32}(-1)F_{01}(-1) + 18 E_{32}^2(-1)F_{01}(-1) E_{31}(-1)H_{01}(-1)\\
& & & -2 a^3 E_{32}(-1)F_{01}(-1) +3abE_{32}(-1)F_{01}(-1)\\
& & & +6aE_{31}(-1)E_{11}(-1)E_{32}(-1)F_{01}(-1)-6vE_{32}(-1)F_{01}(-1)\\
& & & +3aE_{32}(-1)E_{10}(-1)E_{31}(-1)H_{01}(-1)-3vE_{31}(-1)  \\
& & & + 6 u (E_{31}(-1)^2F_{10}(-1) + E_{32}(-1)E_{31}(-1)F_{11}(-1) \}.\mathbf{1} .
\end{align*}
Furthermore,
\[ [[F_{32}(1),u],w]. \mathbf{1} \equiv [[F_{32}(1),v],w]. \mathbf{1} \equiv [[F_{32}(1),w],w]. \mathbf{1} \equiv 0,\]
mod $C_2(N(k,0))$.
\end{Lem}

In the proof of Lemma \ref{lem:P3:universal}, we defined for each $y\in \mathcal{B}_{PBW}$ a projection $\pi_y:\mathcal{U}(\hat{\mathfrak g}_-) \rightarrow \mathbb{C}y$, which sends $y$ to itself and every other basis element to zero.  Let us also denote by $\pi_y(\cdot)$ the induced projection on $N(k,0)$.  I.e., $\pi_y: N(k,0) \rightarrow \mathbb{C} y.\mathbf{1}$ is defined by:
\[
\pi_y(f.\mathbf{1}) = \pi_y(f).\mathbf{1}, \mbox{ for all } f \in \mathcal{U}(\hat{\mathfrak g}_-).
\]
In the final lemma of this appendix, we will compute for certain basis elements $y$ the value of $\pi_y(\cdot)$ on the element $F_{32}(1).(u^p v^q w^r.\mathbf{1}) \in N(k,0)$ for all $p,q,r \in \mathbb{Z}$.  Let us denote
\[
y(p',q') = E_{31}(-1)^{p'+2q'+2} E_{11}(-1)^{p'}  E_{01}(-1)^{q'} F_{10}(-1) \quad (p',q' \in \mathbb{Z}_{\ge 0}).
\]
We will then also write $\pi_{p',q'} = \pi_{y(p',q')}$ for all $p', q' \in \mathbb{Z}_{\ge 0}$.

\begin{Lem}\label{lem:app:a}
For $p',q' \in \mathbb{Z}_{\ge 0}$, let $\pi_{p',q'}: N(k,0) \rightarrow \mathbb{C} y(p',q').\mathbf{1}$ be the projection defined  as above.  If $p,q,r \in \mathbb{Z}_{\ge 0}$ and $r>0$, then
\begin{equation*}
\pi_{p',q'}\left(F_{32}(1).(u^p v^q w^r.\mathbf{1})\right) = 0.
\end{equation*}
On the other hand if $p,q \in \mathbb{Z}_{\ge 0}$, then
\begin{equation*}
\pi_{p',q'}(F_{32}(1).(u^p v^q.\mathbf{1}))= 
\begin{cases}
-2\binom{p}{2}  (-1)^{p+q-2}3^{q}\ y(p-2,q).\mathbf{1} & \text{if $p=p'+2$ and $q = q'$}\\
\phantom{-} 6\binom{q}{2} (-1)^{p+q-1}3^{q-2}\ y(p+1,q-2).\mathbf{1} & \text{if $p=p'-1$ and $q = q'+2$}\\
\phantom{-} 0 &\text{otherwise.}
\end{cases}
\end{equation*}
\end{Lem}

\begin{proof}
Since by Lemma \ref{lem:app:a2} $\big[[F_{32}(1),u ],w \big].\mathbf{1} \equiv \big[[F_{32}(1),v ],w \big].\mathbf{1} \equiv \big[[F_{32}(1),w ],w \big].\mathbf{1} \equiv 0$, mod $C_2(N(k,0))$, it follows from (\ref{eq:c2space:a1}) that
\begin{equation*}
 \tbinom{p}{1}\tbinom{r}{1}  u^{p-1} v^{q} w^{r-1} \big[[F_{32}(1),u ],w \big].\mathbf{1}\ +\ \tbinom{q}{1}\tbinom{r}{1}  u^{p} v^{q-1} w^{r-1} \big[[F_{32}(1),v ],w \big].\mathbf{1}
 \end{equation*}
\begin{equation*} 
   +\  \tbinom{r}{2} u^{p} v^{q} w^{r-2} \big[[F_{32}(1),w ],w \big].\mathbf{1} \quad \equiv \quad 0,
 \end{equation*}
mod $C_2(N(k,0))$.  Thus by Lemma \ref{lem:c2space} it follows that 
\begin{align}
 F_{32}(1). (u^p v^q w^r. \mathbf{1}) \equiv \phantom{LLLLLLLL}& \nonumber\\
  \ \tbinom{p}{1} \ u^{p-1} v^{q} w^{r}  [F_{32}(1), u ].\mathbf{1} & \phantom{L} + \ \ \tbinom{p}{2}\ u^{p-2} v^{q} w^{r} \big[ [F_{32}(1), u ], u \big].\mathbf{1} \label{eq:sing:a1}\\
  +\ \ \tbinom{q}{1}\  u^{p} v^{q-1} w^{r} [F_{32}(1), v ].\mathbf{1} & \phantom{L} +\  \tbinom{p}{1}\tbinom{q}{1}\ u^{p-1} v^{q-1} w^{r} \big[ [F_{32}(1), u ], v \big].\mathbf{1} \nonumber \\
   +\ \tbinom{r}{1}\ u^{p} v^{q} w^{r-1} [F_{32}(1),w ].\mathbf{1} & \phantom{L}  + \  \ \tbinom{q}{2}\ u^{p} v^{q-2} w^{r} \big[ [F_{32}(1), v ], v \big].\mathbf{1}, \nonumber
\end{align}
mod $C_2(N(k,0))$.

Combining (\ref{eq:sing:a1}) with (\ref{eq:c2space:a1}), (\ref{eq:c2space:a2}), and (\ref{eq:c2space:a3}), we see that
\begin{equation}\label{eq:sing:a2}
 F_{32}(1). (u^p v^q w^r. \mathbf{1}) \equiv 0,\quad \mbox{mod } C_2(N(k,0)),
\end{equation}
whenever $r \ge 2$.  Similarly, if $r=1$ we have
\begin{equation}\label{eq:sing:a3}
F_{32}(1).(u^p v^q w.\mathbf{1}) \equiv u^p v^q [F_{32}(1), w].\mathbf{1},
\end{equation}
and in the case $r=0$ we have
\begin{align}
 F_{32}(1). (u^p v^q . \mathbf{1}) \equiv \phantom{LLLLLLLL}& \nonumber\\
  \ \tbinom{p}{1} \ u^{p-1} v^{q}  [F_{32}(1), u ].\mathbf{1} & \phantom{L} + \ \ \tbinom{p}{2}\ u^{p-2} v^{q}  \big[ [F_{32}(1), u ], u \big].\mathbf{1} \label{eq:sing:a4}\\
   & \phantom{L} +\  \tbinom{p}{1}\tbinom{q}{1}\ u^{p-1} v^{q-1}  \big[ [F_{32}(1), u ], v \big].\mathbf{1} \nonumber \\
   +\ \ \tbinom{q}{1}\  u^{p} v^{q-1}  [F_{32}(1), v ].\mathbf{1} & \phantom{L}  + \  \ \tbinom{q}{2}\ u^{p} v^{q-2}  \big[ [F_{32}(1), v ], v \big].\mathbf{1}, \nonumber
\end{align}
mod $C_2(N(k,0))$, respectively.

Next, notice that only two of the formulas listed in the statement of Lemma \ref{lem:app:a2} contain either $F_{10}(-1)$ or $F_{11}(-1)$ as a factor of one of the terms appearing on the right-hand side of the congruence relation.  In the formula for $\big[ [F_{32}(1), u ], u \big].\mathbf{1}$, we find the terms 
\[-2E_{31}(-1)^2F_{10}(-1).\mathbf{1} - 2E_{32}(-1)E_{31}(-1) F_{11}(-1).\mathbf{1}.\]
While in the formula for $\big[ [F_{32}(1), v ], v \big].\mathbf{1}$, we see the terms
\[6u E_{31}(-1)^2F_{10}(-1).\mathbf{1} +6 u E_{32}(-1)E_{31}(-1) F_{11}(-1).\mathbf{1}.\]
And there are no other appearance of $F_{10}(-1)$ or $F_{11}(-1)$ in any of the formulas.  We thus see from (\ref{eq:sing:a1}) and (\ref{eq:sing:a2}) that 
\begin{equation}\label{eq:sing:a5}
\pi_{p',q'}\big( F_{32}(1). (u^p v^q w^r. \mathbf{1} ) \big) \equiv 0, \quad \mbox{mod } C_2(N(k,0)),
\end{equation}
whenever $r \ge 1$.  Similary in the case $r=0$, we may use (\ref{eq:sing:a3}) along with Lemma \ref{lem:app:a2} to verify 
\begin{equation}\label{eq:sing:a6}
\pi_{p',q'}(F_{32}(1).(u^p v^q.\mathbf{1}))\equiv
\begin{cases}
-2\binom{p}{2}  (-1)^{p+q-2}3^{q}\ y(p-2,q).\mathbf{1} & \text{if $p=p'+2$ and $q = q'$}\\
\phantom{-} 6\binom{q}{2} (-1)^{p+q-1}3^{q-2}\ y(p+1,q-2).\mathbf{1} & \text{if $p=p'-1$ and $q = q'+2$}\\
\phantom{-} 0 &\text{otherwise,}
\end{cases}
\end{equation}
mod $C_2(N(k,0)$, respectively.  Now since the map $\pi_{p',q'}$ is a projection onto a one-dimensional subspace of $N(k,0)$, which is spanned by $y(p',q').\mathbf{1}$ and does not vanish modulo the $C_2$-subspace, it follows that the congruence relations in (\ref{eq:sing:a5}) and (\ref{eq:sing:a6}) may both be replaced by equality in $N(k,0)$.  This completes the proof.
\end{proof}

\bigskip 

\section{}
We provide here some results which are used in the proof of Section 4.  Let us begin by recalling some lemmas contained in the previous paper.

\begin{Lem}\cite{AL,P} \label{lem:product} 
Let $X \in \frak g$ and let $Y_1,\dots, Y_m \in \mathcal{U}(\frak g)$.  Then \[(X^n)_L(Y_1 \dots Y_m)= \sum_{\substack{(k_1, \dots, k_m)\in(\mathbb{Z}_{\ge 0})^m\\ \sum{k_i=n}} }\binom{n}{k_1 \dots k_m}(X^{k_1})_L Y_1\dots (X^{k_m})_L Y_m,\]
where $\binom{n}{k_1\dots k_n} = \frac{n!}{k_1!\cdots k_m!}$.
\end{Lem}

\begin{Lem}\label{lem:app:b1} \cite{AL} \hfill
\begin{enumerate}
\item $(E_{ij}^m)_L (F_{ij}^m) \in m!H_{ij}(H_{ij}-1)\cdots(H_{ij}-m+1)+\mathcal{U}(\frak g)E_{ij}$, for all $i\alpha + j\beta \in \Delta_+$.
\item Suppose $X \in  \mathcal{U}(\frak g)_0$, the zero-weight subspace of $\mathcal U(\frak g)$.  Then $X \in \frak n_{-}\ \mathcal{U}(\frak g)$ if and only if $X \in \mathcal{U}(\frak g) \frak n_{+}$. 
\item  For $Y\in \mathcal{U}(\frak g)$ and $n>r>0$, we have \[(E_{ij}^n)_L (F_{ij}^r Y) \in F_{ij}\mathcal{U}(\frak g) + \tfrac{n!}{(n-r)!} (H_{ij}-n+r)\cdots(H_{ij}-n+1)(E_{ij}^{n-r})_L Y +  \mathcal{U}(\frak g)E_{ij}. \] 
\end{enumerate}
\end{Lem}

\begin{Lem} \label{lem:app:b2}\cite{AL}  The following identities hold in $\mathcal{U}(\frak g)$.
First we have:
 \begin{align*}
  (F_{31})_L [a] &=\  F_{10},\\
 \tfrac{1}{2!}(F_{31}^2)_L [b] &=\ F_{21}E_{01} - F_{31}E_{11},\\
 \tfrac{1}{3!}(F_{31}^3)_L [c] &=\  F_{31}(H_{32}+1)E_{01}+F_{32}E_{01}^2-F_{31}^2 E_{32},\\
 (F_{31}^2)_L [a] &= \ (F_{31}^3)_L [b]\ =\ (F_{31}^4)_L [c]= 0.
\end{align*}
 Next we have: 
\begin{align*}
  (F_{32}) _L [a] &= \ F_{11},\\
\tfrac{1}{2!} (F_{32}^2)_L [b] &= \ F_{32}E_{10} - F_{21}F_{01},\\
\tfrac{1}{3!} (F_{32}^3)_L [c] &= \ F_{32}^2 E_{31} -F_{32}F_{01}(H_{31}+2) - F_{31}F_{01}^2 , \\
 (F_{32}^2)_L [a] &= \ (F_{32}^3)_L [b]\ =\ (F_{32}^4)_L [c]= 0.
\end{align*}
\end{Lem}

\begin{Lem} \label{lem:app:b3}\cite{AL}
Suppose that $n,p,q,r \in \mathbb{Z}_{\ge 0}$ and $n=p+2q+3r$.  Then the following hold in $\mathcal{U}(\frak g)$:
\begin{align*}
(E_{10}^n F_{31}^{n})_L([a]^p[b]^q[c]^r) \ \equiv  \left(\tfrac{n!}{(n-p)!}\right)^2 (H_{10}-n+1) \cdots (H_{10}-n+p) (E_{10}^{n-p}F_{31}^{n-p})_L ([b]^q[c]^r) , \\
(E_{11}^n F_{32}^{n})_L([a]^p[b]^q[c]^r) \  \equiv \left(\tfrac{n!}{(n-p)!}\right)^2 (H_{11}-n+1) \cdots (H_{11}-n+p) (E_{11}^{n-p}F_{32}^{n-p})_L ([b]^q[c]^r)    ,
\end{align*}
mod $\mathcal{U}(\mathfrak{g})\mathfrak{n}_+$, respectively.
\end{Lem}

Let us briefly introduce some notation which will help to simplify the statement of the final lemma.  Recall the PBW grading on $\mathcal{S}(\mathfrak h)$ (see e.g. \cite{D}):
\[
\mathcal{S}(\mathfrak h)\ =\ \mathcal{S}(\mathfrak h)^0 \oplus \mathcal{S}(\mathfrak h)^1 \oplus \mathcal{S}(\mathfrak h)^2 \oplus \cdots,
\]
where $\mathcal{S}(\mathfrak h)^0 = \mathbb{C} 1$, $\mathcal{S}(\mathfrak h)^1= \mathfrak h$, and $\mathcal{S}(\mathfrak h)^{n+1} = \mathfrak h\ \mathcal{S}(\mathfrak h)^n$ for $n \in \mathbb{Z}_{ \ge 1}$.  Given $f \in \mathcal{S}(\mathfrak h)$, we write $\mathrm{deg}\ f = n$, if 
$f \in \oplus_{j=0}^n \mathcal{S}(\mathfrak h)^j$ but $f \notin \oplus_{j=0}^{n-1}\mathcal{S}(\mathfrak h)^j$.

Let $\pi_0: \mathfrak g = \mathfrak n_- \oplus \mathfrak h \oplus \mathfrak n_+ \rightarrow \mathfrak h$ be the projection which acts as identity on $\mathfrak h$ and sends $\mathfrak n_{\pm}$ to zero.  
  Also denote by $\pi_0$ the corresponding induced homomorphism $\pi_0: \mathcal{U}(\mathfrak g) \rightarrow \mathcal{S}(\mathfrak h)$.  
\begin{Rmk}\label{rmk:app:b}
Suppose $f \in \mathcal{U}(\mathfrak g)$ and $wt(f)=0$, where $wt(\cdot)$ denotes weight of a homogeneous element with respect to the grading by the root lattice $Q$.  Then $\pi_0(f)$ is the unique polynomial in $\mathcal{S}(\mathfrak h)$ which is congruent to $f$ modulo $\mathcal{U}(\mathfrak g) \mathfrak n_+$.
\end{Rmk}

We now give the final lemma of the appendix.
\begin{Lem}\label{lem:app:b}
Suppose $m \in \mathbb{Z}_{\ge 1},$ $q,r \in \mathbb{Z}_{\ge 0},$ and $m=2q+3r$.  Then the following hold in $\mathcal{S}(\mathfrak h)$:
\begin{enumerate}
\item  $\phantom{L} \pi_0\big( (E_{10}^{m} F_{31}^{m})_L ([b]^q [c]^r) \big) = 0$, \medskip

\item  $\phantom{L} \mathrm{deg}\ \pi_0\big( (E_{11}^{m} F_{32}^{m})_L ([b]^q [c]^r) \big)\  \leq\ m-1$.
\end{enumerate}
\end{Lem}

\begin{proof}
Since we have $(F_{31}^3)_L[b] = (F_{31}^4)_L[c] = 0$ by Lemma \ref{lem:app:b2}, it follows from Lemma \ref{lem:product} that
\begin{equation*}
(F_{31}^m)_L( [b]^q [c]^r ) = \frac{m!}{(2!)^q (3!)^r } \left(  (F_{31}^2)_L [b] \right)^q \left( (F_{31}^3)_L [c] \right)^r.
\end{equation*}
From Lemma \ref{lem:app:b2}, we also have $\tfrac{1}{2!}(F_{31}^2)_L [b] \in \mathcal{U}(\mathfrak g)\mathfrak n_+$ and $\tfrac{1}{3!}(F_{31}^3)_L [c] \in \mathcal{U}(\mathfrak g)\mathfrak n_+$.  Since $n>0$, either $q>0$ or $r>0$.  Thus $(F_{31}^m)_L\left( [b]^q [c]^r \right)\equiv 0$, mod $\mathcal{U}(\mathfrak g)\mathfrak n_+$.  The first part of the lemma then follows from Remark \ref{rmk:app:b}.

To prove the second equation, we first have $(F_{32}^3)_L[b] = (F_{32}^4)_L[c]=0$, by Lemma \ref{lem:app:b2} as before.  We then also have
\begin{equation*}
(F_{32}^3)_L ([b]^q [c]^r) = \frac{m!}{(2!)^q(3!)^r} \left(  (F_{32}^2)_L [b]  \right)^q \left(  (F_{32}^3)_L [c]\right)^r,
\end{equation*}
by Lemma \ref{lem:product}.  Before proceeding further with this case, we will need an additional fact.

Suppose $m' \in \mathbb{Z}_{\ge 1}$; $m' \leq m$; and $X_1, \dots, X_{m'} \in \mathcal{U}(\mathfrak g)$ are basis elements from (\ref{eq:basis}) such that 
\begin{equation}\label{eq:app:b}
wt(E_{11}^m X_1 \cdots X_{m'}) = 0,
\end{equation} 
with respect to the root lattice $Q$.  Then it is not difficult to see that
\begin{equation*}
\mathrm{deg}\  \pi_0\big( (E_{11}^m)_L (X_1 \cdots X_{m'})\big) \leq m,
\end{equation*}
and equality holds if and only if $m'=m$ and $X_1 = \cdots = X_{m} = F_{11}$.  

From the preceding paragraph, it is then sufficient to show  the product
 \begin{equation}\label{eq:app:b1}
 \left(\tfrac{1}{2!}(F_{32}^2)_L [b] \right)^q \left( \tfrac{1}{3!}(F_{31}^3)_L [c]\right)^r
\end{equation}
can be written as a linear combination of terms  $X_{1} \cdots X_{m'} \neq F_{11}^m$, where $X_{1}, \dots, X_{m'}$ satisfy condition (\ref{eq:app:b}).  
From Lemma \ref{lem:app:b2} we have
\begin{equation}\label{eq:app:b2}
\tfrac{1}{2!}(F_{32}^2)_L [b] = F_{32}E_{10} - F_{21}F_{01}
\end{equation}
and
\begin{align}
\tfrac{1}{3!}(F_{31}^3)_L [c] &=\ F_{32}^2 E_{10} - F_{32}F_{01}(H_{31}+2) - F_{31} F_{01}^2 \nonumber \\ 
& =\ F_{32}^2E_{31} - 2F_{32}F_{01} - \tfrac 1 3  F_{32}F_{01}H_{10}  -  \tfrac 1 3  F_{32}F_{01}H_{21} - F_{31}F_{01}^2\label{eq:app:b3}, 
\end{align}
since $H_{31} = H_{10}+H_{01} = \tfrac 1 3 (H_{10} + H_{21})$.  Now each term appearing in the linear combination (\ref{eq:app:b2}) (resp. (\ref{eq:app:b3})) satisfies (\ref{eq:app:b}) with $m=2$ (resp. $m=3$).  It is then clear that the product (\ref{eq:app:b1}) may be written as a linear combination of terms $X_1\cdots X_{m'}$ ($\neq F_{11}^m$) of the form (\ref{eq:app:b}) with $m=2q+3r$.
\end{proof}

 \begin{Rmk}
We note a correction to \cite{AL}.  On page 221, the first congruence in Lemma B.14 should read:
\[ (E_{21}^4 F_{21}^8)_L ([b]^2) \equiv 4! 8! (2 H_{21}(H_{11}-1) + 2H_{10}(H_{10}-1) - 6H_{01}(H_{01}+1)),\]
mod $\mathcal{U}(\mathfrak g) \mathfrak n_+$.  Similarly on page 222, the  congruence appearing in the last line of the proof of Lemma B.14 should also be changed accordingly.
\end{Rmk}

\bibliographystyle{siam}

\end{document}